\documentclass[11pt]{article}
\usepackage{amssymb}
\usepackage[latin1]{inputenc}
\usepackage{graphicx,color}
\def\nd{\noindent}

\usepackage{color}
\newtheorem{theorem}{Theorem}[section]

\newtheorem{lemma}{Lemma}[section]
\newtheorem{proposition}{Proposition}[section]

\newtheorem{corollary}{Corollary}[section]
\newcommand{\fim}{\hfill\rule{2mm}{2mm}}

\date{}
\begin{document}
\title{
\vspace{0.5in} {\bf\Large  Blow-up solutions for a
 $p$-Laplacian\\ elliptic equation of logistic type with singular
nonlinearity }}

\author{
{\bf Claudianor O. Alves}\footnote{C.O. Alves was partially supported by CNPq/Brazil 301807/2013-2}\\
{\small \textit{Unidade Acad\^emica de Matem\'atica}}\\
{\small \textit{Universidade Federal de Campina Grande}}\\
{\small \textit{58429-900, Campina Grande - PB - Brazil}}\\
{\small \textit{e-mail address: coalves@dme.ufcg.edu.br}}\\
{\bf\large Carlos A. Santos}\footnote{The author acknowledges
the support of PROCAD/UFG/UnB  and FAPDF under grant PRONEX
193.000.580/2009}\,\, {\bf\large and \,\,  Jiazheng
Zhou}\footnote{The author acknowledges the support of CNPq/Brasil.
}\hspace{2mm}\\
{\it\small Departamento de Matem\'atica}\\
{\it\small Universidade de Bras\'ilia}\\
{\it\small 70910-900 Bras\'ilia, DF - Brasil}\\
{\it\small e-mails: csantos@unb.br,
jiazzheng@gmail.com }\vspace{1mm}}
\date{}
\maketitle \vspace{-0.2cm}

\begin{abstract}
\noindent In this paper, we deal with existence, uniqueness and exact rate of boundary
behavior of blow-up solutions for a class of
logistic type quasilinear problem in a smooth bounded domain involving
the $p$-Laplacian operator, where the nonlinearity  can have a
singular behavior. In the proof of the existence of solution, we have used the sub
and super solution method in conjunction with variational techniques
and comparison principles. Related to the rate on
boundary and uniqueness, we combine a comparison principle proved in the present paper together
with the our result of existence of solution.
\end{abstract}

\nd {\it \footnotesize 2012 Mathematics Subject Classifications:} {\scriptsize  35A15,  35B44.  35H30 }\\
\nd {\it \footnotesize Key words}: {\scriptsize Variational methods,
Blow-up solution, Logistic type, Quasilinear equations.}

\section{Introduction}
\def\theequation{1.\arabic{equation}}\makeatother
\setcounter{equation}{0}

\mbox{}

In this paper, we consider existence, uniqueness and exact rate of boundary
behavior of blow-up (large or explosive) solutions
for the following class of quasilinear problem of logistic type
$$
\left\{
\begin{array}{l}
-\Delta_p{u}=\lambda a(x)g(u)-b(x)f(u) \,\,\, \mbox{in} \,\, \Omega, \\
u>0 \,\,\, \mbox{in} \,\,\, \Omega, ~~ u=+\infty \,\,\, \mbox{on}
\,\,\, \partial \Omega,
\end{array}
\right. \eqno{(P)_{\lambda}}
$$
where $\Omega \subset \mathbb{R}^{N}$ is a bounded domain with
smooth boundary, $ \lambda>0$ is a parameter, $\Delta_p$ stands
for the $p$-Laplacian operator given by $ \Delta_p{u}=div(|\nabla
u|^{p-2}\nabla u),~1<p<+\infty $, $a,b\in L^{\infty}(\Omega)$ are
appropriated functions with $b\gneqq0$, $a$ can change of sign and
$f:[0,+\infty) \to [0,+\infty)$, $g:(0,+\infty)\to [0,+\infty)$ are
continuous functions verifying some technical conditions, which will
be fixed later on. Here, we are principally interested in the case where $g$ is singular at $0$, that is,
$g(s)\to +\infty$ as $s \to 0^+$.

We say that a function $u \in C_{loc}^{1,\nu}(\Omega)$, for some
$\nu \in (0,1)$, is a solution of  problem $(P)_{\lambda}$, if
$$
u(x) \to  +\infty \,\,\, \mbox{as} \,\,\, d(x):=dist(x, \partial
\Omega)\to 0
$$
and
$$
\int_{\Omega}|\nabla u|^{p-2}\nabla u \nabla \varphi = \int_{\Omega}
[\lambda a(x)g(u)-b(x)f(u)] \varphi, ~\mbox{for all}~\varphi \in
C_{0}^{\infty}(\Omega),
$$
where $d(x)$ stands for the distance of point $x \in \Omega$ to $\partial \Omega$.

Next, we made an overview about this class of problem. In 2002,
 Delgado, L\'opez-G\'omez and Su\'arez \cite{delgado-gomez-suarez0} showed existence of blow-up solution for the  problem
$$
\left\{
\begin{array}{l}
-\Delta{u}=\lambda u^{{1}/{m}}-b(x)u^{{p}/{m}} \,\,\, \mbox{in} \,\, \Omega, \\
u>0 \,\,\, \mbox{in} \,\,\, \Omega, ~~ u=+\infty \,\,\, \mbox{on}
\,\,\, \partial \Omega,
\end{array}
\right.
$$
where $\lambda \in \mathbb{R}, 1<m<p$ and $b(x) \geq 0$.

Motivated by that paper, in 2004, the same authors studied in
\cite{delgado-gomez-suarez} the ensuing problem
$$
\left\{
\begin{array}{l}
-\Delta{u}=a(x)u^{q}-b(x)f(u) \,\,\, \mbox{in} \,\, \Omega, \\
u>0 \,\,\, \mbox{in} \,\,\, \Omega, ~~ u=+\infty \,\,\, \mbox{on}
\,\,\, \partial \Omega,
\end{array}
\right.
$$
where $0<q<1$, $a \in L^{\infty}(\Omega)$, $0<b\in
C^{\mu}(\overline{\Omega})$ for some $\mu \in (0,1)$ and $f$
satisfies some technical conditions, such as, $f$ is an increasing continuous and verifies the Keller-Osserman condition,
that is,
$$
\int_{1}^{\infty} {F}(t)^{-{1}/{p}} dt < \infty \eqno{(KO)}
$$
where $F(t)=\int_{0}^{t}f(\tau)d\tau$.

In 2006, the same class of problem was considered by Du
\cite{Du}, with $q=1$ and $f(u)=u^{p}$. In 2009, Feng in \cite{Feng}
showed that the problem
$$
\left\{
\begin{array}{l}
 -\Delta u=\lambda g(u)-b(x)f(u) \ \ \mbox{in}\ \ \Omega,\\
 u>0 \,\,\, \mbox{in} \,\,\, \Omega, ~~
u=+\infty\ \ \mbox{on}\ \ \ \partial\Omega,
\end{array}
\right.
$$
admits an unique solution for $\lambda\in\mathbb{R}$, $0<b\in
C^\mu(\overline{\Omega})$ for some $\mu\in(0,1)$ and $f,g$ being increasing continuous functions satisfying additional conditions.

As an exception to the previous works, in 2010, Wei in \cite{wei}
studied the problem $(P)_\lambda$ with negative exponents, more precisely, the singular problem
$$
\left\{
\begin{array}{l}
-\Delta{u}=a(x)u^{-m}-b(x)u^{p}\,\, \mbox{in} \,\, \Omega, \\
u>0 \,\,\, \mbox{in} \,\,\, \Omega, ~~ u=+\infty \,\,\, \mbox{on}
\,\,\, \partial \Omega,
\end{array}
\right.
$$
where $p>1, m>0$ and $a,b \in C^{\mu}(\overline{\Omega})$ for
some $\mu \in (0,1)$ with $b$ being a positive function.

Related to quasilinear problems, in 2012, Wei and Wang \cite{lei-wang}
worked with the ensuing quasilinear boundary problem
$$
\left\{
\begin{array}{l}
-\Delta_p{u}=a(x)u^{m}-b(x)u^{q}\,\, \mbox{in} \,\, \Omega, \\
u>0 \,\,\, \mbox{in} \,\,\, \Omega, ~~ u=+\infty \,\,\, \mbox{on}
\,\,\, \partial \Omega,
\end{array}
\right.
$$
where $0<m<p-1<q$, $a \in L^{\infty}(\Omega)$ and  $b$ being a
non-negative function.

One year later, in \cite{chen-wang}, Chen and Wang improved the results found in \cite{lei-wang}, because they showed that the problem
$$
\left\{
\begin{array}{l}
 -\Delta_p u=a(x)g(u)-b(x)f(u),\ \ \ \mbox{in}\ \ \Omega,\\
u=+\infty\ \ \mbox{on}\ \ \ \partial\Omega,
\end{array}
\right.
$$
has a  solution, supposing that $a\in L^\infty(\Omega)$, $b\in
C^\mu(\bar{\Omega})$ for some $0<\mu<1$, $b(x)\geq0$, $g$ is a nondecreasing and nonnegative continuous function with
$g(0)=0$ and $f\in C^1$ is an increasing function with $f(0)=0$ and  $f(s)>0$
for $s>0$. Moreover, $f(s)$ grows more
slowly than $s^q$ with $q>p-1$ and  $g(s)$ does not grow faster than
$s^{p-1}$ at infinity.

Concerning the boundary behavior, in 2006, Ouyang and Xie
\cite{Ouyang} established a blow-up rate of the large positive
solutions of the problem
$$
\left\{
\begin{array}{l}
 -\Delta u=\lambda u-b(\Vert x - x_0\Vert
 )u^q \ \ \mbox{in}\  B,\\
u=+\infty\ \ \mbox{on} \ \ \partial B,
\end{array}
\right.
$$
where $B=B_R(x_0)$ stands for the ball centered at $x_0 \in
\mathbb{R}^N$ with radius $R$, $b:[0,R] \to (0,\infty)$ is a
continuous function, $q>1$ and $\lambda\in\mathbb{R}$. Under
additional conditions on $b$, they obtained a  rate of boundary
behavior accurate  of the unique solution for the above problem.

In 2009, Feng \cite{Feng} obtained the exact asymptotic behavior and uniqueness of solution for the problem
$$
\left\{
\begin{array}{l}
 -\Delta u=\lambda g(u)-b(x)f(u) \ \ \mbox{in}\  \Omega,\\
u=+\infty\ \ \mbox{on} \ \ \partial \Omega,
\end{array}
\right.
$$
where $\Omega\subset \mathbb{R}^N$ is a smooth bounded domain, $\lambda\in\mathbb{R}$, $0 \leq b\in C^\mu(\bar{\Omega})$ for some
$\mu\in(0,1)$, $b=0$ on $\partial\Omega$ and there exists an increasing positive function $h\in C^1(0,\delta_0)$ for some $\delta_0>0$ verifying
$$
\lim_{d(x)\to0^+}{b(x)}/{h^2(d(x))}=c_0>0, \,\,\, \lim_{r\to0^+}{\Big(\int_0^rh(s)ds\Big)}/{h(r)}=0
$$
and
$$
\lim_{r\to0^+}\Big[{\Big(\int_0^rh(s)ds\Big)}/{h(r)}\Big]'=l_1>0.
$$
Related to $f$ and $g$, it was assumed that $0 \leq f,g\in C^1([0,+\infty))$, $f(0)=0$, $f'\geq0$, $f'(0)=0$, ${f(s)}/{s}$, $s>0$ increasing, $f$ is $RV_q $ with $q>1$; $ g$ increasing with
$\lim_{s\to0^+}g'(s)>0$, ${g(s)}/{s}$, $s>0$ in non-increasing and $g$ belongs to $RV_q $ with $0<q<1$. In that paper, an arbitrary function $h:[s_o,\infty) \to (0, \infty)$, for some $s_0>0$, belongs to class $RV_q$, for some $q \in \mathbb{R}$, if
$$
\lim_{s\to \infty}h(ts)/h(s)=t^q \,\,\, \mbox{for all} \,\,\, t>0.
$$

Still in 2009, Meli\'an \cite{Garcia-Melian}  established an exact boundary behavior and uniqueness for the problem
$$
\left\{
\begin{array}{l}
 \Delta_p u=b(x)u^q \ \ \mbox{in}\  \Omega,\\
u=+\infty\ \ \mbox{on} \ \ \partial \Omega,
\end{array}
\right.
$$
where $\Omega \subset \mathbb{R}^N$ is a smooth bounded domain, $q>p-1>0$ with $b$ satisfying
$$
\lim_{x\to x_0}d(x)^{\gamma(x)}b(x)=Q(x_0) \,\,\, \mbox{for} \,\,\, x_0\in\partial\Omega,
$$
for some $\gamma\in C^\mu(\bar{\Omega})$ with $0<\mu<1$, $\gamma (x) < 0$ and $Q(x)>0$ for all $x \in \partial\Omega$.

In 2012, Li, Pang and Wang \cite{Li-Pang-Wang} also showed the boundary behavior and uniqueness for the problem
 $$
\left\{
\begin{array}{l}
 -\Delta_p u=a(x)u^m-b(x)f(u) \ \ \mbox{in}\  \Omega,\\
u=+\infty\ \ \mbox{on} \ \ \partial \Omega,
\end{array}
\right.
$$
where $\Omega\subset\mathbb{R}^N$ is a smooth bounded domain, $0<m<p-1$, $0\leq a\in L^\infty(\Omega)$,
$f\in C^1([0,\infty))\cap RV_q $, for some $q>p-1$,  with
$f(0)=0$ and $f(s)>0$ for $s>0$, $f(s)/s^{p-1}$, $s>0$  increasing,
$b\in C^\mu(\bar{\Omega})$ for some $0<\mu<1$ with $b\geq0$, $b(x)\not\equiv0$ in $\Omega$,
$\bar{\Omega}_0=\{x\in\bar{\Omega}~/~ b(x)=0\}\subset\Omega$ is a non-empty and connected set with $C^2$-boundary and some additional conditions on
$b$.

In the same year,  Chen and Wang \cite{chen-wang} proved the boundary behavior and uniqueness of solution for the problem
 $$
\left\{
\begin{array}{l}
 -\Delta_p u=a(x)g(u)-b(x)f(u) \ \ \mbox{in}\  \Omega,\\
u=+\infty\ \ \mbox{on} \ \ \partial \Omega,
\end{array}
\right.
$$
where $\Omega\subset\mathbb{R}^N$ is a smooth bounded domain with $N\geq2$, $p>1$, $0 \leq a\in L^\infty(\Omega)$,
$g\in C([0,\infty))\cap RV_q$, for some $q \leq p-1$ with $g$ being  nondecreasing, $g(s)/s^{p-1}$, $s>0$  nonincreasing function satisfying $g(0)=0$, $f$ is such that $f(0)=0$, $f(s)>0$ for $s>0$ and $f(s)/s^{p-1}$, $s>0$  increasing,
$0 \leq b\in C^\mu(\bar{\Omega})$ for some $0<\mu<1$ with  $b\not\equiv0$ in $\Omega$ and more hypotheses on $b$, $f$ and
$g$.

Still in 2012, Xie and Zhao \cite{xie-zhao} established the uniqueness and
the blow-up rate of the large positive solution of the quasilinear
elliptic problem
$$
\left\{
\begin{array}{l}
 -\Delta_pu=\lambda u^{p-1}-b(\Vert x - x_0\Vert
 )f(u) \ \ \mbox{in}\  B,\\
u=+\infty\ \ \mbox{on} \ \ \partial B,
\end{array}
\right.
$$
 where $N\geq2$, $2\leq p<\infty$,
$\lambda>0$ is a parameter and the weight function $b:[0,R] \to
(0,\infty)$ is a continuous function satisfying additional
assumptions. Moreover, $f$ is a locally Lipschitz continuous function with
${f(s)}/{s^{p-1}}$ increasing for $s \in (0,+\infty)$ and $f(s)\sim s^q$
for large $s>0$ with $q>p-1$.

Motivated principally by the above papers and their results, we will
study  existence and uniqueness of blow-up solutions and the exact boundary behavior rate.
To do that, we fix
$$
\displaystyle a_0=essinf_{\Omega} a ,~~~~~~~~ b_0=essinf_{\Omega}b ~~~~~~~~~~~~~~~~
$$
and assume that $f$ satisfies $(KO)$ and the conditions below
\begin{enumerate}
\item[$(f_0)$]$~~~~$ $\displaystyle\liminf_{s \to +\infty}\frac{\inf\Big\{ \frac{f(t)}{t^{p-1}}, \,\,\, t \geq s \Big\}}{f(s)/s^{p-1}}>0,$
\item[$(f_1)$]$~~~~$ $(i)~~~~$ $\displaystyle\lim_{s \to 0^+}\frac{f(s)}{s^{p-1}}=0$
\hspace{.2cm} \hspace{.2cm} $~(ii)$ $~~~~\displaystyle\lim_{s \to
+\infty}\frac{f(s)}{s^{p-1}}=+\infty$.
\end{enumerate}

Associated with $g:(0,+\infty) \to (0,+\infty)$, we assume that
\begin{enumerate}
\item[$(g_0)$]$~~~~$ $(i)~~~~$ $\displaystyle\lim_{s \to 0^+}\frac{g(s)}{s^{p-1}}=l \in [0,+\infty]$
\hspace{.2cm} \hspace{.2cm} $~(ii)$ $~~~~\displaystyle\lim_{s \to
+\infty}\frac{g(s)}{s^{p-1}}<+\infty$.
\end{enumerate}

Our main result is the following
\begin{theorem} \label{T1}
Assume $a,b \in L^{\infty}(\Omega) $ with $a_0>-\infty$ and $b_0>0$.  If $f$ satisfies $(KO)$, $(f_0)$,
$(f_1)$ and $(g_0)$ holds, then there exist $ \lambda_* \in (0,+\infty]$ and a real number $\sigma_o>0$ such that the problem $(P)_{\lambda}$ has a solution $u=u_{\lambda}\geq \sigma_o$ for
each $0 < \lambda < \lambda_*$ given. Moreover, 
$\lambda_*=+\infty$ if $a_0\geq 0$.
\end{theorem}

Related to condition $(f_0)$, it is important to observe that:
\begin{enumerate}
  \item [{$\bf{(i)}$}] if $f(s)/s^{p-1}$, $s \geq s_0$ is nondecreasing, for some $s_o>0$, then the limit at
  $(f_0)$ is equal to 1,
  \item [{$\bf{(ii)}$}] if $$f(t)= \left\{\begin{array}{c}
                             \sigma(t) t^{p-1}, 0<t<1,\\
                              t^{p-1}e^{-t~},~t\geq 1,
                           \end{array}
  \right.
  $$
  where $\sigma \geq 0$ is a continuous function satisfying
  $\sigma(1)=e^{-1}$ and $\displaystyle \lim_{t\to0}\sigma(t)=0$, then the limit
  at $(f_0)$ is null and $f$ does not satisfy (KO). This example shows the necessity of hypothesis $(f_0)$,
\end{enumerate}

To state our next result, we need consider other assumptions on $f$
and $g$, more specifically
\begin{enumerate}
\item[$(f_1)^{\prime}$]$~~~~$ $0<\displaystyle\lim_{t\to+\infty}\frac{f(t)}{t^q}=f_{\infty}<+\infty \ \mbox{for some}\
q>p-1$
\item[$(g_0)^{\prime}$]$~~~~$ $0\leq\displaystyle\lim_{t\to+\infty}\frac{g(t)}{t^m}=g_{\infty}<+\infty \ \mbox{for some}\
m\leq p-1$
\end{enumerate}
and concerning the continuous potentials $a$ and $b$, we will suppose
\begin{enumerate}
\item[$(a)$]  there exists $R\in C(\bar{\Omega})$ with
$R(x)\geq0$  on $U_{\delta}$, for some $\delta>0$ such that
$$\lim_{x\to x_0} d(x)^{\eta(x)}a(x)=R(x_0)\ \mbox{for each}\ x_0\in\partial\Omega,
$$
where
$$
\eta(x)=\frac{(p-1-m)}{q-p+1}(p-\gamma(x))+p \,\,\, \forall x\in \overline{\Omega} \,\,\, \mbox{and} \,\,\, U_{\delta}:=\{x\in \Omega~/~d(x)<\delta\}.
$$
\item[$(b)$]  there exist $Q\in C(\bar{\Omega})$ and $\gamma\in C^\mu(\bar{\Omega})$, for some $0<\mu<1$, such that
$$
\lim_{x\to x_0}d(x)^{\gamma(x)}b(x)=Q(x_0),\ \mbox{for each}\ x_0\in
\partial\Omega;
$$
with $Q(x)>0$, $x\in \partial\Omega$ and $\gamma(x)\leq0$ for all $ x \in U_{\delta}$.

\end{enumerate}

Related to above notations, we have the ensuing result.

\begin{theorem}\label{TP2}
Assume that $a,b\in L^{\infty}(\Omega)$ with $a \geq 0$ a.e. on $U_{\delta}$, for some $\delta>0$. If $(f_1)^{\prime}$, $(g_0)^{\prime}$, $(a)$ and $(b)$ hold and $u\in C^1(\Omega)$ is a positive solution of
${(P)_{\lambda}}$, then
$$\lim_{x\to x_0}d(x)^{\alpha(x)}u(x)=A(x_0),\ \mbox{for each}\ x_0\in
\partial\Omega,$$
where $\alpha(x)=({p-\gamma(x)})/({q-p+1})$, $x \in \Omega$ and
$A(x_0)$ is the unique positive solution of
$$f_{\infty}Q(x_0)A^{q-m}(x_0)-(p-1)\alpha(x_0)^{p-1}(1+\alpha(x_0))A^{p-m-1}(x_0)-\lambda g_{\infty}R(x_0)=0,~x_0\in
\partial\Omega .$$
Moreover, if $a_0,b_0 \geq 0$, ${f(t)}/{t^{p-1}}$ is nondecreasing
and ${g(t)}/{t^{p-1}}$ is nonincreasing for $t \in (0,+\infty)$, then the problem $(P)_{\lambda}$ has at most one solution.
\end{theorem}

Related to assumptions $(g_0)^{\prime}$, we would like to detach that if $g$ is $(p-1)$-sublinear at infinite, that is,  $g_{\infty}=0$ with  $m=p-1$, the behavior of the solution is unaffected by  $g$.

 As an immediate consequence of our results, we have the following corollary
\begin{corollary}\label{C2}
Assume $a\in  L^{\infty}(\Omega)$ and $b \in C(\overline{\Omega})$ with $a_0=0$ and $b_0>0$. If \linebreak $-\infty<m\leq p-1< q$, for each
$\lambda >0$ the quasilinear  problem
$$
\left\{
\begin{array}{l}
-\Delta_p{u}= \lambda a(x)u^{m}-b(x)u^{q} \,\,\, \mbox{in} \,\, \Omega, \\
u>0 \,\,\, \mbox{in} \,\,\, \Omega, ~~ u=+\infty \,\,\, \mbox{on}
\,\,\, \partial \Omega,
\end{array}
\right. 
$$
has a unique solution $u=u_{\lambda}\in C^1(\Omega)$  satisfying
$$
\lim_{x\to x_0}d(x)^{{p}/({q-p+1})}u(x)= b(x_0)^{-1/(q-p+1)}\Big(\frac{(p-1)(q+1)p^{p-1}}{(q-p+1)^p}\Big)^{-1/(q-p+1)},
\ \mbox{for each}\ x_0\in
\partial\Omega.
$$
\end{corollary}
\begin{proof}. By the hypotheses on $a$ and $b$, we can choose $R(x)=0$ and $\gamma(x)=0$ for  $x \in \overline{\Omega}$. Then,
$$
\alpha(x)={p}/({q-p+1}), \,\,\, \mbox{for all} \,\,\, x \in \overline{\Omega} \,\,\, \mbox{and} \,\,\, Q(x_0)=b(x_0) \,\,\, \mbox{for each} \,\,\,  x_0 \in \partial\Omega.
$$
Moreover, by the hypotheses on $f$ and $g$, we also have
$f_\infty=1=g_\infty$. Hence, by Theorem \ref{T1}, it follows the existence of a solution $u \in C^1(\Omega)$ of $(P)_{\lambda}$ and by Theorem \ref{TP2}, it follows the uniqueness and
$$
A(x_0)= b(x_0)^{-1/(q-p+1)}\Big(\frac{(p-1)(q+1)p^{p-1}}{(q-p+1)^p}\Big)^{-1/(q-p+1)},
\ \mbox{for each}\ x_0\in \partial\Omega.
$$
This concludes the proof of this corollary. \fim
\end{proof}

\vspace{0.5 cm}

We would like point out that the main contributions of our results for this class of problem  are the following:\\

\noindent 1- They complement  and improve some results found in the literature, because we permit that the nonlinearity $g$ may behave as a decreasing and/or singular function. \\

\noindent 2- The existence, exact boundary behavior and uniqueness results were obtained by assuming a set of hypotheses more general than those considered in the literature up to now. \\

We now briefly outline the organization of the contents of this
paper. In Section 2, by using a sub and supersolution method in
conjunction with variational method, we prove the existence of
solution for two auxiliary blow-up problems. Section 3 is devoted to
prove the existence of blow-up solution for $(P)_\lambda$, while in Section 4 we study the rate boundary of the solutions.

\section{Auxiliary problem}

In this section, we are interested in the existence of solution for the
ensuing quasilinear problem
$$
\left\{
\begin{array}{l}
-\Delta_p{u}=\lambda a(x)g(u)-b(x) f(u), \,\,\, \mbox{in} \,\, \Omega, \\
u>0 \,\,\, \mbox{in} \,\,\, \Omega, ~~ u=L \,\,\, \mbox{on} \,\,\,
\partial \Omega,
\end{array}
\right. \eqno{(P)_{L}}
$$
where $L>1$ is an appropriated real number. Associated with above problem, we have the following result.

\begin{proposition} \label{T3}
Assume $a,b \in L^{\infty}(\Omega)$ with $a_0>-\infty$ and $b_0>0$. If $(f_1)$ and
$(g_0)$ hold, then there exist  $\lambda_* \in (0, +\infty]$ and $\sigma_o>0$, which does not depend on $L>0$, such that
$(P)_L$ has a solution $u = u_{\lambda,L}\in
C^{1}(\overline{\Omega})$ for each $0
< \lambda < \lambda_*$ and $L>L_0$, for some $L_0>0$. Moreover, 
$ u(x) \geq \sigma_o>0$ for all $x \in \overline{\Omega}$ and  $\lambda_*=+\infty$ if  $a_0\geq 0$.
\end{proposition}

The proof of this proposition is based on the lemmas below.

\begin{lemma} \label{T2}
Assume $a,b \in L^{\infty}(\Omega)$ with $a_0>-\infty$ and $b_0>0$.  If $(f_1)$ and
$(g_0)$ hold, then there exist $\lambda_* \in (0, +\infty]$, $L_0>0$ and a  $\sigma_o>0$, which does not depend on $L>0$, such that $(P)_{L}$ has a sub solution
$\underline{u} = \underline{u}_{\lambda,L}\in C^1(\overline{\Omega})$,  for each $L>L_0$ and  $\lambda \in (0, \lambda_*)$ given.
Furthermore, $\underline{u}(x) \geq \sigma_o$ for all $x \in \overline{\Omega}$ and $\lambda_*=+\infty$, if  $a_0\geq 0$.
\end{lemma}
\noindent \begin{proof} In the sequel, we will divide our proof into two cases. 

\vspace{0.5 cm}

\noindent {\bf First Case:  $ a_0 \geq 0$.} \,\, From $(f_1)$, the function
$$
\tilde{f}(s)=s^{p-1}\sup\left\{\frac{f(t)}{t^{p-1}}, \,\, t \leq s
\right\}+s^p \,\,\, \mbox{for} \,\,\, s\in (0,+\infty),
$$
is  continuous and verifies
$$
\begin{array}{l}
(i)  \,\,\,\, \displaystyle \frac{\tilde{f}(s)}{s^{p-1}},~s>0 \,\,\, \mbox{is increasing} ~~~~~~ (ii)~~\tilde{f}(s) \geq f(s),~s>0\\
(iii)\,\,\,\, \displaystyle \lim_{s \to
0^+}\frac{\tilde{f}(s)}{s^{p-1}}=0 ~~~~~~ (iv)~~ \lim_{s \to
+\infty}\frac{\tilde{f}(s)}{s^{p-1}}=+\infty.
\end{array}
$$

Now, considering the problem
\begin{equation}\label{binfty}
 \left\{
\begin{array}{l}
-\Delta_p{v}=-\|b\|_\infty\tilde{f}(v)\,\,\, \mbox{in} \,\, \Omega, \\
0<v\leq 1\ \ \ \mbox{in}\ \ \Omega, \,\,\, v=1\ \ \mbox{on} \,\,\, \partial \Omega,
\end{array}
\right.
\end{equation}
it follows from Theorem 1.1 in \cite{Diaz} that there exists a solution $\underline{u} \in C^1{(\overline{\Omega})}$   of the problem (\ref{binfty}). Besides this, the  positivity of $\underline{u}$ is a consequence of strong maximum principle of Vazquez and $v\leq 1$ in $\Omega$ follows from the standard comparison principle.

So, $\underline{u}$ satisfies
$$
\left\{
\begin{array}{l}
-\Delta_p{u}\leq\lambda a(x)g(u)-b(x) f(u) \,\,\, \mbox{in} \,\, \Omega, \\
u\geq \gamma_1>0 \,\,\, \mbox{in} \,\,\, \overline{\Omega}, ~~ u\leq L \,\,\, \mbox{on} \,\,\,
\partial \Omega,
\end{array}
\right.
$$
for all $L\geq 1$ and $\lambda>0$ given, where $\gamma_1=\min_{\bar{\Omega}}\underline{u}>0$.

\medskip

\noindent {\bf Second Case:} $a_0 < 0$. Applying again $(f_1)$ and ($g_0$), the function
$$
 \hat{g}(s)=s^{p-1}\sup\left\{\frac{g(t)}{t^{p-1}}, \,\, t \geq s
\right\}+1, ~s>0,
$$
is continuous and verifies 
$$
\begin{array}{l}
(i)~~ \displaystyle \frac{\hat{g}(s)}{s^{p-1}},~s>0~\mbox{is decreasing }~~~~(ii)~~\hat{g}(s) >g(s),~s>0\\
(iii)~~\displaystyle \lim_{s \to
0^+}\frac{\hat{g}(s)}{s^{p-1}}=\infty ~~~~(iv)~~ \lim_{s \to
+\infty}\frac{\hat{g}(s)}{s^{p-1}}<\infty.\\
\end{array}
$$
Next, we denote by  $w \in C^{1,\mu}(\overline{\Omega})$ the unique positive solution of the problem
$$
\left\{
\begin{array}{l}
\Delta_p{u}=u^{p-1} \,\,\, \mbox{in} \,\, \Omega, \\
0 < u \leq 1 \,\,\, \mbox{in} \,\,\, \Omega, ~~u=1 \,\,\, \mbox{on} \,\,\,
\partial \Omega.
\end{array}
\right. \eqno{(P_3)}
$$
The existence of the above function can be found in \cite{Diaz}. Defining $w_0=\displaystyle \min_{\overline{\Omega}}w>0$
 and
$$
\varphi(M)={\frac{(Mw_0)^{p-1}}{\hat{g}(Mw_0)}}\Big[1-b_\infty\frac{\tilde{f}(M)}{M^{p-1}}\Big] \,\,\, \mbox{for} \,\,\, M>0,
$$
we see that 
$$
\displaystyle\lim_{M\to\infty}\varphi(M)=-\infty,\ \
\displaystyle\lim_{M\to0}\varphi(M)\geq0~\mbox{and}~\varphi(\tilde{M})>0,~\mbox{for
some}~\tilde{M}>0.
$$
Thereby, there is $M_0>0$ such that
$$
\varphi(M_0)=\sup\{\varphi(M)~/~M>0\}.
$$
In the sequel, we denote by $\lambda_* $ the real number given by
$$
\lambda_* := - \displaystyle\frac{\sup\{\varphi(M)~/~M>0\}}{a_0}=-\displaystyle\frac{\varphi(M_0)}{a_0}>0.
$$
Thus, for  $\lambda \in (0, \lambda_*)$,
$$
1-b_\infty\frac{\tilde{f}(M_0)}{M_0^{p-1}}\geq (-\lambda a_0)\frac{\hat{g}(M_0w_0)}{(M_0w_0)^{p-1}}.
$$
Now, remembering that $\displaystyle {\tilde{f}(s)}/{s^{p-1}}$
is increasing and $\displaystyle {\hat{g}(s)}/{s^{p-1}}$ is
decreasing in the interval $(0,+\infty)$, we obtain
$$
1-b_\infty\frac{\tilde{f}(M_0w(x))}{(M_0w(x))^{p-1}}\geq (-\lambda a_0)\frac{\hat{g}(M_0w(x))}{(M_0w(x))^{p-1}},~ \forall x \in \Omega.
$$
Taking $\underline{u}=M_0w\geq M_0
w_0:=\gamma_2>0$ and using the last inequality, we see that
$$
\left\{
\begin{array}{l}
-\Delta_p(\underline{u})\leq  \lambda a_0 \hat{g}(\underline{u})-
b_\infty \tilde{f}(\underline{u}) \leq\lambda a(x)g(\underline{u})-b(x)f(\underline{u}) \,\,\, \mbox{in} \,\,\, \Omega,\\
\mbox{}\\
\underline{u} \geq \gamma_2~\mbox{in}~\Omega,~~\underline{u} < L
\,\,\, \mbox{on} \,\,\,
\partial \Omega,
\end{array}
\right.
$$
for all $L>M_0$. Hence, choosing $\sigma_o=\min\{\gamma_1,\gamma_2\}>0$ and
$L_0=\max\{1,M_0\}$, we get the desired result. \fim
\end{proof}

\vspace{0.5 cm}

For the super solution, our result is the following.

\begin{lemma} \label{T4}
Assume $a,b \in L^{\infty}(\Omega) $ with $a_0>-\infty$ and $b_0>0$. If $(f_1)$ and
$(g_0)$ hold true, then  $\overline{u}(x):=L\in C^1(\overline{\Omega})$
 is a super solution of $(P)_{L}$ for each $L>L_0$, where $L_0$ was given in Lemma $\ref{T2}$.
Moreover, $\underline{u} \leq \overline{u}$.
\end{lemma}

\noindent \begin{proof} Let $\lambda \in (0,\lambda_*)$, where $\lambda_*>0$ was given in Lemma 2.1. By $(f_1)-(ii)$ and
$(g_0)-(ii)$, we can choose $0<c_3<c_4$ and $t_\infty>1$  positive constants verifying
\begin{equation} \label{E_1}
f(t) \geq c_3t^{p-1} \,\,\, \mbox{and} \,\,\,\, g(t)
\leq c_4t^{p-1} \,\,\, \mbox{for all} ~ t \in (t_\infty,+\infty)
\end{equation}
and
$$
\lambda c_4 a_{\infty}- {b_0c_3}<0.
$$
Defining $\overline{u}=L$, with $L \geq \max\{L_0, t_{\infty}\}$, we derive
$$
-\Delta_p{\overline{u}}=0 >(\lambda a_{\infty}c_4-b_0c_3)L^{p-1}
\mbox{} \geq \lambda
a_{\infty}g(\overline{u})-b_0f(\overline{u}) \,\,\,
\mbox{in} \,\,\, \Omega.
$$
Consequently, $\overline{u} \in C^1(\overline{\Omega})$ and it satisfies
$$
\left\{
\begin{array}{l}
-\Delta_p{\overline{u}}\geq  \lambda a(x)g(\overline{u})-b(x)f(\overline{u}) \,\,\, \mbox{in} \,\,\, \Omega,\\
\overline{u}\geq\underline{u} ~\mbox{in}~\Omega,~~\overline{u} \geq
L \,\,\, \mbox{on} \,\,\,
\partial \Omega,
\end{array}
\right.
$$
for all $L \geq \max\{L_0, t_{\infty}\}$. This finishes the proof.
\fim.
\end{proof}

\subsection{\bf Proof of Proposition \ref{T3}} As a consequence of Lemmas 2.1 and 2.2, the functions
$$
\underline{v}=\underline{u}-L \,\,\,\, \mbox{and} \,\,\,\ \overline{v}=\overline{u}-L
$$
are sub and super solution respectively of the problem
$$
\left\{
\begin{array}{l}
-\Delta_p{v}=\lambda a(x)g(v+L)-b(x)f(v+L), \,\,\, \mbox{in} \,\, \Omega, \\
\mbox{}\\
v=0 \,\,\, \mbox{on} \,\,\, \partial \Omega,
\end{array}
\right.
\eqno{(P_1)}
$$
for each $L>0$ large enough.

Hereafter, we will consider the function $h:\Omega \times \mathbb{R} \to \mathbb{R}$ given by
$$
h(x,t)=
\left\{
\begin{array}{l}
\lambda a(x)g(\underline{v}+L)-b(x) f(\underline{v}+L), \,\,\, \mbox{if} \,\,\, t \leq \underline{v}(x),\\
\lambda a(x)g(t+L)-b(x)f(t+L), \,\,\, \mbox{if} \,\,\, \underline{v}(x) \leq t \leq \overline{v}(x),\\
\lambda a(x)g(\overline{v}+L)-b(x)f(\overline{v}+L), \,\,\, \mbox{if} \,\,\, t \geq \overline{v}(x)\\
\end{array}
\right.
$$
and the problem
$$
\left\{
\begin{array}{l}
-\Delta_p{v}=h(x,v)\,\,\, \mbox{in} \,\, \Omega, \\
v=0 \,\,\, \mbox{on} \,\,\, \partial \Omega.
\end{array}
\right.
\eqno{(P_2)}
$$

Our goal is proving  that problem $(P_2)$ has a weak solution $v$ satisfying
$$
\underline{v}(x) \leq v(x) \leq \overline{v}(x) \,\,\,\, \mbox{a.e. in} \,\,\, \Omega,
$$
because the above estimate gives that $v$ is a weak solution of $({P_1})$.

We observe that the energy functional associated with the above problem is given by
$$
I(v)=\frac{1}{p}\int_{\Omega}|\nabla v|^{p}-\int_{\Omega}H(x,v) \,\,\,\, \forall v \in W_{0}^{1,p}(\Omega),
$$
where $H(x,t)=\int_{0}^{t}h(x,\tau) d \tau$.

A direct calculus shows that $I$ belongs to $C^{1}(W_{0}^{1,p}(\Omega),\mathbb{R})$ with
$$
I'(v)\phi=\int_{\Omega}|\nabla v|^{p-2}\nabla v\nabla \phi - \int_{\Omega}h(x,v)\phi \,\,\,\, \forall \phi \in W_{0}^{1,p}(\Omega).
$$
Hence, $ v \in   W_{0}^{1,p}(\Omega)$ is a weak solution of $(P_2)$ if, and only if, $v$ is a critical point of $I$.

It is easy to check that $I$ is weak s.c.i and boundedness from below in $W_0^{1,p}(\Omega)$. Then, there is $v_0 \in W_0^{1,p}(\Omega)$ such that
$$
I'(v_0)=0\,\,\, \mbox{and} \,\,\, I(v_0)=\min\{I(v) \,\, : \,\, v \in W_0^{1,p}(\Omega)\},
$$
and so, $v_0$ is a weak solution of $(P_2)$. Next, we will show that 
\begin{equation} \label{E4}
\underline{v}(x) \leq v_0(x) \leq \overline{v}(x) \,\,\, \mbox{a.e. in} \,\,\, \Omega.
\end{equation}
Considering the test function
$\phi=(v_0-\overline{v})^{+}$, we find that
$$
I'(v_0)\phi=0
$$
or equivalently
$$
\int_{\Omega}|\nabla v_0|^{p-2}\nabla v_0 \nabla (v_0-\overline{v})^{+}=\int_{\Omega}h(x,v_0)(v_0-\overline{v})^{+}.
$$
Thus, by definition of $h$,
$$
\int_{\Omega}|\nabla v_0|^{p-2}\nabla v_0 \nabla (v_0-\overline{v})^{+}=\int_{\Omega}(\lambda a(x)g(\overline{v}+L)-
b(x) f(\overline{v}+L))(v_0-\overline{v})^{+}.
$$
Using that $\overline{v}$ is a super solution of $(P_1)$, we get the inequality
$$
\int_{\Omega}|\nabla v_0|^{p-2}\nabla v_0 \nabla (v_0-\overline{v})^{+} \leq \int_{\Omega}|\nabla \overline{v}|^{p-2}\nabla \overline{v}
\nabla (v_0-\overline{v})^{+},
$$
and so,
$$
\int_{\Omega}\left\langle |\nabla v_0|^{p-2}\nabla v_0-|\nabla \overline{v}|^{p-2}\nabla \overline{v}, \nabla v_0 - \nabla
\overline{v} \right\rangle \leq 0.
$$
Since $-\Delta_p$ is a strictly monotone operator, the last inequality implies that
$$
(v_0-\overline{v})^{+}=0,
$$
leading to
$$
v_0(x)\leq \overline{v}(x) \,\,\, \mbox{a.e in} \,\,\, \Omega.
$$
The same type of arguments can be used to prove that
$$
\underline{v}(x) \leq v_0(x) \,\,\, \mbox{a.e in} \,\,\, \Omega.
$$
Now, we observe that (\ref{E4}) follows of last two inequalities. Setting $u=v_0+L$, we  have that it is a solution of
$(P)_{L}$ with
\begin{equation} \label{E41}
\underline{u}(x) \leq u(x) \leq \overline{u}(x) \,\,\,\, \mbox{a.e. in} \,\,\, \Omega,
\end{equation}
obtaining the desired result. \fim

\section{Proof of the Theorem \ref{T1}}

In this section, we will finish the proof Theorem \ref{T1}. To do that, we will need of two auxiliary results below. The first one is due to Matero \cite{matero} and it has the ensuing statement

\begin{lemma} \label{L1}
Assume that $\Omega$ is a smooth bounded domain in $\mathbb{R}^N$ and
$h:(0,\infty)\to (0,\infty)$ is a continuous and increasing function
satisfying $(KO)$. Then, the quasilinear problem
$$
\left\{
\begin{array}{l}
\Delta_p{u}= h(u) \,\,\, \mbox{in} \,\, \Omega, \\
 u=+\infty \,\,\, \mbox{on}
\,\,\, \partial \Omega,
\end{array}
\right.
$$
admits a positive solution $u \in C^1(\Omega)$.
\end{lemma}

Next, we show a comparison result, which is crucial in our approach, and its proof follows by adapting some arguments found in \cite{chen-wang-1}, \cite{Lindqvist}. However, for the reader's convenience, we will write its proof.

\begin{lemma} \label{L2}
$(Comparison ~Principle)$. Suppose that $\Omega$ is a bounded domain
in $\mathbb{R}^N$ and that \linebreak $\alpha,\beta:\Omega \to [0,\infty)$ are
nonnegative continuous functions. Let $u_1,\ u_2\in C^1(\Omega)$ be positive functions verifying, in the sense of distribution,
$$
\left\{
\begin{array}{l}
-\Delta_pu_1 \geq \alpha(x)h(u_1)-\beta(x)k(u_1) \,\,\, \mbox{in} \,\, \Omega, \\
-\Delta_pu_2\leq \alpha(x)h(u_2)-\beta(x)k(u_2) \,\,\, \mbox{in} \,\,\, \Omega, \\
\limsup_{d(x,\partial\Omega)\to0}(u_2-u_1)\leq0,
\end{array}
\right.
$$
where $h, k:[0,\infty) \to [0,\infty)$ are continuous functions. If for $s\in(\inf_\Omega\{u_1,u_2\},\sup_\Omega\{u_1,u_2\})$
\begin{enumerate}
  \item [$(a)$] ${h(s)}/{s^{p-1}}$ is decreasing, ${k(s)}/{s^{p-1}}$ is non-decreasing and $\alpha
\in{L^\infty(\Omega)}$ with $\alpha\not\equiv0$,
\end{enumerate}	
or
\begin{enumerate}
  \item [$(b)$] ${h(s)}/{s^{p-1}}$ is non-increasing and ${k(s)}/{s^{p-1}}$ is increasing and $\beta
\in{L^\infty(\Omega)}$ with $\beta\not\equiv0$
\end{enumerate}
 holds, then $u_1\geq u_2$ in $\Omega$.
\end{lemma}
\begin{proof} By hypothesis,
\begin{equation}
\label{30}
-\int_\Omega[|\nabla u_2|^{p-2}\nabla u_2\nabla\varphi_2-|\nabla u_1|^{p-2}\nabla u_1\nabla\varphi_1]\geq\int_\Omega\alpha(x)[h(u_1)
\varphi_1-h(u_2)\varphi_2]+\int_\Omega\beta(x)[k(u_2)\varphi_2-k(u_1)\varphi_1],
\end{equation}
for all $0 \leq \varphi_1,\ \varphi_2 \in C_0^\infty(\Omega)$. Then, by density, we can consider the functions $v_1, v_2\in W_0^{1,p}(\Omega)$ given by
$$
v_1=\frac{[(u_2+{\epsilon}/{2})^p-(u_1+\epsilon)^p]^+}{(u_1+\epsilon)^{p-1}}~~\mbox{and}~~
 v_2=\frac{[(u_2+{\epsilon}/{2})^p-(u_1+\epsilon)^p]^+}{(u_2+{\epsilon}/{2})^{p-1}} \,\,\,\, \mbox{with} \,\,\,\, \epsilon>0
$$
as test functions. Now, denoting by $\Omega_\epsilon$ the set defined by
$$
\Omega_\epsilon=\{x\in\Omega~/~ u_2(x)+{\epsilon}/{2}>u_1(x)+\epsilon\},
$$
we have $\overline{\Omega}_\epsilon \subset \Omega_0:=\{x\in\Omega~/~ u_2(x)>u_1(x)\}\subset\Omega$,
$$\nabla v_1=-\Big[1+(p-1)\Big(\frac{u_2+{\epsilon}/{2}}{u_1+\epsilon}\Big)^p\Big]\nabla u_1+p\Big(\frac{u_2+{\epsilon}/{2}}{u_1+\epsilon}\Big)^{p-1}
\nabla u_2,$$
and
$$
\nabla v_2=\Big[1+(p-1)\Big(\frac{u_1+\epsilon}{u_2+{\epsilon}/{2}}\Big)^p\Big]\nabla u_2-p\Big(\frac{u_1+\epsilon}{u_2+{\epsilon}/{2}}\Big)^{p-1}
\nabla u_1~~\mbox{in}~~\Omega_\epsilon .
$$
Hence,
\begin{equation}
\label{31}
\begin{array}{lll}
I&:=&|\nabla u_2|^{p-2}\nabla u_2\nabla v_2-|\nabla u_1|^{p-2}\nabla u_1\nabla v_1\\
\\
 &=&|\nabla u_2|^{p-2}\nabla u_2[1+(p-1)(\frac{u_1+\epsilon}{u_2+{\epsilon}/{2}})^p]\nabla u_2-p(\frac{u_1+\epsilon}{u_2+
{\epsilon}/{2}})^{p-1}|\nabla u_2|^{p-2}\nabla u_2\nabla u_1\\
\\
&&+|\nabla u_1|^{p-2}\nabla u_1[1+(p-1)(\frac{u_2+{\epsilon}/{2}}{u_1+\epsilon})^p]\nabla u_1-p(\frac{u_2+{\epsilon}/{2}}
{u_1+\epsilon})^{p-1}|\nabla u_1|^{p-2}\nabla u_1\nabla u_2\\
\\
&=&\{[1+(p-1)(\frac{u_1+\epsilon}{u_2+{\epsilon}/{2}})^p]|\nabla u_2|^p+[1+(p-1)(\frac{u_2+{\epsilon}/{2}}{u_1+\epsilon})^p]
|\nabla u_1|^p\}\\
\\
&&-p(\frac{u_1+\epsilon}{u_2+{\epsilon}/{2}})^{p-1}|\nabla u_2|^{p-2}\nabla u_1\nabla u_2-p(\frac{u_2+{\epsilon}/{2}}{u_1+
\epsilon})^{p-1}|\nabla u_1|^{p-2}\nabla u_1\nabla u_2~~\mbox{in}~~\Omega_\epsilon.
\end{array}
\end{equation}
Now, setting  $w_1=u_1+\epsilon$ and $w_2=u_2+{\epsilon}/{2}$, it follows that
$$
V_1:=\nabla \ln(w_1)=\frac{\nabla u_1}{u_1+\epsilon}=\frac{\nabla u_1}{w_1}
$$
and
$$
V_2:=\nabla \ln(w_2)=\frac{\nabla u_2}{u_2+{\epsilon}/{2}}=\frac{\nabla u_2}{w_2}~~\mbox{in}~~\Omega_\epsilon.
$$
Putting $V_1$ and $V_2$  in (\ref{31}), we get
\begin{equation}
\label{32}
\begin{array}{lll}
 I&=&\{w_2^p|V_2|^p+(p-1)w_1^p|V_2|^p+w_1^p|V_1|^p+(p-1)w_2^p|V_1|^p\}\\
\\
&&-pw_1^p|V_2|^{p-2}V_1V_2-pw_2^p|V_1|^{p-2}V_1V_2\\
\\
&=&w_2^p|V_2|^p-w_1^p|V_2|^p+w_1^p|V_1|^p-w_2^p|V_1|^p\\
\\
&&+pw_1^p|V_2|^p+pw_2^p|V_1|^p-pw_1^p|V_2|^{p-2}V_1V_2-pw_2^p|V_1|^{p-2}V_1V_2\\
\\
&=&w_2^p(|V_2|^p-|V_1|^p-p|V_1|^{p-2}V_1(V_2-V_1))+w_1^p(|V_1|^p-|V_2|^p-p|V_2|^{p-2}V_2(V_1-V_2)).
\end{array}
\end{equation}
Then, from (\ref{32}) and  \cite[Lemma 4.2]{Lindqvist},
$$
\begin{array}{lll}
 I&\geq&c(p)\frac{|V_2-V_1|^2}{(|V_2|+|V_1|)^{2-p}}w_2^p+c(p)\frac{|V_1-V_2|^2}{(|V_1|+|V_2|)^{2-p}}w_1^p=
 c(p)(w_1^p+w_2^p)\frac{|V_1-V_2|^2}{(|V_1|+|V_2|)^{2-p}},
\end{array}
$$
if $1<p<2$, where $c(p)$ is a real positive constant depending on just $p$. Moreover,
$$\begin{array}{lll}
   I&\geq&w_2^p\frac{|V_2-V_1|^p}{2^{p-1}-1}+w_1^p\frac{|V_1-V_2|^p}{2^{p-1}-1}=\frac{1}{2^{p-1}-1}(w_1^p+w_2^p){|V_1-V_2|^p},
  \end{array}
$$
if $p\geq 2$. Gathering the above information,
\begin{equation}\label{I}
I\geq C(p)(w_1^p+w_2^p)\frac{|V_1-V_2|^{p+(2-p)^+}}{(|V_1|+|V_2|)^{(2-p)^+}}~\mbox{for all}~ p>1,
\end{equation}
for some $C(p)$ positive.

Now, (\ref{I}) combined with (\ref{30}) gives
\begin{equation}
\label{33}
 \begin{array}{lll}
         \displaystyle C(p)\int_{\Omega_\epsilon}(w_1^p+w_2^p)\frac{|V_1-V_2|^{p+(2-p)^+}}{(|V_1|+|V_2|)^{(2-p)^+}} & + & \displaystyle\int_{\Omega_\epsilon}
\alpha(x)\Big[\frac{h(u_1)}{w_1^{p-1}}-\frac{h(u_2)}{w_2^{p-1}}\Big](w_2^p-w_1^p)\\
\\
            & \leq & \displaystyle\int_{\Omega_\epsilon}\beta(x)\Big[\frac{k(u_1)}{w_1^{p-1}}-
\frac{k(u_2)}{w_2^{p-1}}\Big](w_2^p-w_1^p).
         \end{array}
\end{equation}
Since $\Omega_\epsilon\subset \subset \Omega_0\subset\Omega$, we know that
$$
\frac{k(u_1)}{w_1^{p-1}}-\frac{k(u_2)}{w_2^{p-1}}\leq \frac{k(u_2)}{u_2^{p-1}}
\Big[\frac{u_1^{p-1}}{w_1^{p-1}}-\frac{u_2^{p-1}}{w_2^{p-1}}\Big]<0~ \mbox{in}\ \Omega_\epsilon.
$$
We claim that there exists $K>0$, which does not depend on $\epsilon>0$, such that
$$
\frac{h(u_1)}{w_1^{p-1}}-\frac{h(u_2)}{w_2^{p-1}}\geq-K\ \mbox{in}\ \Omega_\epsilon.
$$
In fact, if the last inequality does not occur, there are $\epsilon_n\in(0,1]$ and $x_n\in\Omega_{\epsilon_n}$ verifying
\begin{equation}
\label{34}\frac{h(u_1(x_n))}{w_1^{p-1}(x_n)}-\frac{h(u_2(x_n))}{w_2^{p-1}(x_n)}\to-\infty~~\mbox{when}~~n \to \infty,
\end{equation}
then, we would have limit ${h(u_2(x_n))}/{w_2^{p-1}(x_n)}\to+\infty$, which leads to
$$
\frac{h(u_2(x_n))}{u_2^{p-1}(x_n)}=\frac{h(u_2(x_n))}{w_2^{p-1}(x_n)}\frac{w_2^{p-1}(x_n)}{u_2^{p-1}(x_n)}\to+\infty,
$$
implying that $u_2(x_n)\to0$, and so,  $x_n\to\partial\Omega$.

Since
$$
\begin{array}{l}
  \displaystyle\frac{h(u_1(x_n))}{w_1^{p-1}(x_n)}-\frac{h(u_2(x_n))}{w_2^{p-1}(x_n)}=
\frac{h(u_1(x_n))}{w_1^{p-1}(x_n)}-\frac{u_2^{p-1}(x_n)}{u_1^{p-1}(x_n)}\frac{h(u_1(x_n))}{w_1^{p-1}(x_n)}+
\frac{u_2^{p-1}(x_n)}{u_1^{p-1}(x_n)}\frac{h(u_1(x_n))}{w_1^{p-1}(x_n)}-\frac{h(u_2(x_n))}{w_2^{p-1}(x_n)} \\
\\
 =\displaystyle\frac{h(u_1(x_n))}{u_1^{p-1}(x_n)}\frac{[u_1^{p-1}(x_n)-u_2^{p-1}(x_n)]}{w_1^{p-1}(x_n)}+
u_2^{p-1}(x_n)\Big[\frac{h(u_1(x_n))}{u_1^{p-1}(x_n)}\frac{1}{w_1^{p-1}(x_n)}-\frac{h(u_2(x_n))}{u_2^{p-1}(x_n)}\frac{1}{w_2^{p-1}(x_n)}\Big],
\end{array}
$$
the limit $\limsup_{x\to\partial\Omega}(u_2-u_1)\leq0$, in conjunction with the fact that $h(s)/s^{p-1}$ is decreasing in $(0,+\infty)$, yields
$$
\liminf_{n\to\infty}\Big[\frac{h(u_1(x_n))}{w_1^{p-1}(x_n)}-\frac{h(u_2(x_n))}{w_2^{p-1}(x_n)}\Big]\geq0,
$$
which is a contradiction with (\ref{34}).

Next, we intend to use Fatou's Lemma in (\ref{33}). However, to do that,  we must to prove that there is $M>0$, which does not depend on $\epsilon \in (0,1)$ such that
$$
0<w_2^p(x)-w_1^p(x)\leq M \,\,\, \forall x \in  \Omega_\epsilon \,\,\, \mbox{and} \,\,\, \forall \epsilon\in(0,1).
$$
Indeed, arguing by contradiction, we assume that there are $\epsilon_n\in(0,1]$ and $x_n\in \Omega_{\epsilon_n}$, such that
$$
M_n=(u_2(x_n)+{\epsilon_n}/{2})^p-(u_1(x_n)+\epsilon_n)^p\to+\infty.
$$
The above limit gives $u_2(x_n)\to+\infty$, and thus, $d(x_n)=d(x_n,\partial \Omega)\to0$. Rewriting $M_n$ as
$$
M_n=\Big(1+\frac{\epsilon_n}{2u_2(x_n)}\Big)^p[u_2(x_n)^p-u_1(x_n)^p]+\Big[\Big(1+\frac{\epsilon_n}{2u_2(x_n)}\Big)^p-
\Big(1+\frac{\epsilon_n}{u_1(x_n)}\Big)^p\Big]u_1(x_n)^p,
$$
the inequality $u_1(x_n)\leq u_2(x_n)$ in $\Omega_{\epsilon_n}$ together with $\limsup_{x\to\partial\Omega}(u_2-u_1)\leq0$ leads to
$$
\limsup_{n\to\infty}M_n\leq0,
$$
obtaining a contradiction.

Now, assume that $(a)$ holds. Using Fatou's Lemma in (\ref{33}), we find
$$
0 \leq C(p)\int_{\Omega_0}(u_1^p+u_2^p)\frac{|\nabla\ln u_1-\nabla\ln u_2|^{p+(2-p)^+}}{(|\nabla\ln u_1|+|\nabla\ln u_2|)^{(2-p)^+}}+
\int_{\Omega_0}\alpha(x)\Big[\frac{h(u_1)}{u_1^{p-1}}-\frac{h(u_2)}{u_2^{p-1}}\Big](u_2^p-u_1^p)\leq0,
$$
from it follows that
$$
\nabla\ln u_1-\nabla\ln u_2\equiv0 \,\,\, \mbox{and} \,\,\, \alpha(x)\equiv0 \,\,\, \mbox{in} \,\,\, \Omega_0,
$$
because ${h(s)}/{s^{p-1}}$ is decreasing in $(0,+\infty)$. Consequently, $u_2=cu_1$
in $\Omega_0$ for some $c>1$, because $u_2(x)>u_1(x)$, $x \in \Omega_0$. Since $\alpha(x)\equiv0$ in $\Omega_0$, we get that $\Omega_0\subsetneq \Omega$, because $\alpha\neq 0$ in $\Omega$. Denoting by $D=\Omega\cap\partial\Omega_0\neq\varnothing$, and taking a open set $\Sigma\subseteq\Omega_0$ such that
$\partial\Sigma\cap D\neq \varnothing$, it follows that $u_1=cu_2$ in $\Sigma$ and $u_1=u_2$ on
$\partial\Sigma\cap D$. So, we must have $c=1$, obtaining a contradiction.

Finally, assuming $(b)$ and applying the Fatou's Lemma in (\ref{33}), we are led to  inequality
$$
0 \leq C(p)\int_{\Omega_0}(u_1^p+u_2^p)\frac{|\nabla\ln u_1-\nabla\ln u_2|^{p+(2-p)^+}}{(|\nabla\ln u_1|+|\nabla\ln u_2|)^{(2-p)^+}}+
\int_{\Omega_0}\beta(x)\Big[\frac{k(u_2)}{u_2^{p-1}}-\frac{k(u_1)}{u_1^{p-1}}\Big](u_2^p-u_1^p)\leq0,
$$
which permits to apply the same arguments as in $(a)$, finishing the proof of Lemma \ref{L2}.
\end{proof}
\fim
\medskip

\noindent{\bf Proof of Theorem \ref{T1}-Completed:} First of all, we
consider  the following auxiliary blow-up problem
$$
\left\{
\begin{array}{l}
-\Delta_p{u}=\lambda \|a\|_{\infty}\hat{g}(u)-b_0 \hat{f}(u) \,\,\, \mbox{in} \,\, \Omega, \\
u>0 \,\,\, \mbox{in} \,\,\, \Omega, ~~ u=+\infty \,\,\, \mbox{on}
\,\,\, \partial \Omega,
\end{array}
\right.\eqno{(P_{6})}
$$
where $\hat{g}$ was fixed in the proof of Lemma \ref{T2} and $\hat{f}$ is given by
$$
\hat{f}(s)=s^{p-1}\inf\left\{\frac{f(t)}{t^{p-1}}, \,\, t \geq s\right\} \,\,\, \mbox{for} \,\,\, s>0.
$$ 
By $(f_1)$ and $(g_0)$, $\hat{f}$ is continuous and verifies 
$$
\begin{array}{l}
(iv)~~ \displaystyle \frac{\hat{f}(s)}{s^{p-1}},~s>0~\mbox{is nondecreasing },~~~~(v)~~\hat{f}(s) \leq f(s),~s>0\\
(vi)~~\displaystyle \lim_{s \to 0^+}\frac{\hat{f}(s)}{s^{p-1}}=0
~~~~~~(vii)~~ \lim_{s \to
+\infty}\frac{\hat{f}(s)}{s^{p-1}}=\infty.
\end{array}
$$

Next, we fix the function $h:(0,+\infty) \to \mathbb{R}$  by
$$
h(t)=b_0\hat{f}(t)-\lambda\|a\|_{\infty}\hat{g}(t)=t^{p-1}
\left[b_0\frac{\hat{f}(t)}{t^{p-1}}-\lambda \|a\|_{\infty}\frac{\hat{g}(t)}{t^{p-1}}
\right].
$$
Using the properties on $\hat{f}$ and $\hat{g}$, we derive that
$$
h(t_0)<0 \,\,\, \mbox{for some} \,\,\, t_0>0, \displaystyle \lim_{t\to +\infty}h(t)=+\infty ,
$$
and that $h$ is increasing in  $(t_1,+\infty)$, where $t_1>0$ is the unique number verifying $h(t_1)=0$, or equivalently,
$$
b_0\frac{\hat{f}(t_1)}{{t_1}^{p-1}}=\lambda\|a\|_{\infty}\frac{\hat{g}(t_1)}{{t_1}^{p-1}}.
$$
Considering $ \tilde{h}(t)=h(t+t_1)$ for $t \in (0,+\infty)$,  we see that $\tilde{h}$ is a continuous, positive and increasing function verifying $(KO)$. In fact, as
$$
\lim_{s\to +\infty}\frac{\hat{f}(s+t_1)}{(s+t_1)^{p-1}}=+\infty
\,\,\, \mbox{and} \,\,\, \lim_{s \to
+\infty}\frac{\hat{g}(s+t_1)}{(s+t_1)^{p-1}}=0,
$$
there exists $s_0>0$ such that
$$
\frac{\hat{g}(s+t_1)}{(s+t_1)^{p-1}}  < \frac{b_0}{2\lambda \|a\|_{\infty}}
\frac{\hat{f}(s+t_1)}{(s+t_1)^{p-1}}~\mbox{for all}~ s >s_0.
$$
Consequently,
$$
\begin{array}{ll}
\tilde{H}(t):=\displaystyle \int_{0}^{t}\tilde{h}(s)ds & =\displaystyle \int_{0}^{s_0}\tilde{h}(s)ds+\int_{s_0}^{t}
 [b_0\hat{f}(s+t_1)-\lambda\|a\|_{\infty}\hat{g}(s+t_1)]ds\\
\mbox{} &> \displaystyle \int_{s_0}^{t}(s+t_1)^{p-1}\Big[ b_0
\frac{\hat{f}(s+t_1)}{(s+t_1)^{p-1}}-
\lambda\|a\|_{\infty} \frac{\hat{g}(s+t_1)}{(s+t_1)^{p-1}}  \Big]ds\\
\mbox{} &> \displaystyle \int_{s_0}^{t}(s+t_1)^{p-1}\frac{b_0}{2} \frac{\hat{f}(s+t_1)}{(s+t_1)^{p-1}} ds \\
\mbox{} &=  \displaystyle \frac{b_0}{2} \int_{s_0}^{t}\hat{f}(s+t_1)ds,~ \forall t>s_0. \\
\end{array}
$$
Thus,
$$
\begin{array}{ll}
\displaystyle \int_{s_0}^{+\infty}\tilde{H}(t)^{-1/p} dt& <
\displaystyle \Big( \frac{2}{b_0}\Big)^{\frac{1}{p}}
\int_{s_0}^{+\infty}{\Big( \int_{s_0+t_1}^
{t+t_1}\hat{f}(\tau)d \tau \Big)^{{-1}/{p}}}{dt} \\
\\
\mbox{} &< \displaystyle \Big( \frac{2}{b_0}\Big)^{\frac{1}{p}}
\int_{s_0}^{+\infty}{\Big( \int_{s_0+t_1}^{t} \hat{f}(\tau)d
\tau \Big)^{{-1}/{p}}}{dt}<+\infty.
\end{array}
$$
Here, we have used ($f_0$) and the fact that $\hat{f}$ verifies $(KO)$.

Therefore, by Lemma \ref{L1}, the the blow-up problem
$$
\left\{
\begin{array}{l}
\Delta_p{u}=\tilde{h}(u) \,\,\, \mbox{in} \,\, \Omega, \\
u=+\infty \,\,\, \mbox{on} \,\,\, \partial \Omega
\end{array}
\right.
$$
admits a solution $\xi \in C^{1}(\Omega)$. Now, defining
$w=\xi+t_1$, we obtain that $w$ is a solution of ($P_6$).

In the sequel, we fix $(L_n) \subset (0,+\infty)$
satisfying $L_n < L_{n+1}$ for all $n \in \mathbb{N}$ with
$L_1=L_0 +1$, where $L_0$ was given in Proposition \ref{T3}. By Proposition \ref{T3}, there exists $u_1\in C^1(\Omega)$
satisfying
$$
\left\{
\begin{array}{l}
-\Delta_p{u_1}=\lambda a(x)g(u_1)-b(x) f(u_1) \,\,\, \mbox{in} \,\, \Omega, \\
u_1>\gamma_0 \,\,\, \mbox{in} \,\,\, \Omega, ~~u_1=L_1 \,\,\,
\mbox{on} \,\,\, \partial \Omega.
\end{array}
\right.
$$
Using  Proposition \ref{T3} together with Lemma \ref{L2}, we find a sequence
$(u_n)_{n\in \mathbb{N}}\subset C^{1}(\overline{\Omega})$ satisfying
$$
\left\{
\begin{array}{l}
-\Delta_p{u_n}=\lambda a(x)g(u_n)-b(x) f(u_n) \,\,\, \mbox{in} \,\, \Omega, \\
u_n \geq u_{n-1}\geq \gamma_0 \,\,\, \mbox{in} \,\,\, \Omega,
~~u_n=L_n \,\,\, \mbox{on} \,\,\, \partial \Omega.
\end{array}
\right.
$$
Gathering the above information, $u_n$ and $w$ satisfy
$$
\left\{
\begin{array}{l}
-\Delta_pu_n\leq\lambda \|a\|_{\infty}\hat{g}(u_n)-b_0\hat{f}(u_n)\ \ \mbox{in}\ \ \Omega, \\
-\Delta_p{w}=\lambda \|a\|_{\infty}\hat{g}(w)-b_0 \hat{f}(w) \,\,\, \mbox{in}
\,\, \Omega,\\
 \limsup_{d(x,\partial\Omega)\to0}(u_n- w)=-\infty\leq0.
\end{array}
\right.
$$
From Lemma \ref{L2},
$$
\gamma_0 < u_1\leq u_2 \leq \cdots \leq u_n\leq u_{n+1}\leq \cdots \leq w.
$$
Now, using standard arguments there is $u \in C^{1}(\Omega)$, such
that $u_n \to u$ in $ C_{loc}^{1}(\Omega)$ and
$$
\left\{
\begin{array}{l}
-\Delta_p{u}=\lambda a(x)g(u)-b(x)f(u) \,\,\, \mbox{in} \,\, \Omega, \\
u>0 \,\,\, \mbox{in} \,\,\, \Omega, ~~u=+\infty \,\,\, \mbox{on}
\,\,\, \partial \Omega.
\end{array}
\right.
$$
 This completes the proof of Theorem \ref{T1}. \fim

 \section{Proof of the Theorem \ref{TP2}}

The proof of Theorem \ref{TP2} is a consequence of the three technical lemmas below. The first of them establishes the behavior of the solution near of the boundary.

\begin{lemma}\label{lemaU}
Assume  $a,b\in L^{\infty}(\Omega)$ satisfy $(a)$, $(b)$, $(f_1)^{\prime}$ and $(g_0)^{\prime}$ hold. If $u\in C^1(\Omega)$ is a solution of
${(P)_{\lambda}}$, then there exist a neighborhood $U_\delta\subset \Omega$
of $\partial\Omega$ and positive constants $c_1, c_2$ such that
$$
c_1d(x)^{-\alpha(x)}\leq u(x)\leq c_2d(x)^{-\alpha(x)},\ x \in
U_\delta,
$$
where $U_{\delta}:=\{x\in \Omega~/~d(x)<\delta\}$ and $\alpha(x)=({p-\gamma(x)})/({q-p+1})$ for all $x \in U_\delta$.  
\end{lemma}

The second one proves an exact rate boundary behavior for an one-dimensional problem.

\begin{lemma}\label{lemau1}
Let  $p \in [m+1,q+1)$ and $\gamma\leq0$ be a real number. If $Q,R>0$ are real constants and $u\in C^1(0,+\infty)$
 is a solution of problem
\begin{equation}\label{u1}
\left\{
\begin{array}{l}
-(|u'|^{p-2}u')'= Rx^{-\eta}u^m - Qx^{-\gamma}u^q, \,\,\, x>0, \\
u>0~\mbox{in}~(0,\infty);~~ u(x)\buildrel
x \to 0 \over \longrightarrow \infty,
\end{array}
\right.
\end{equation}
where $\eta=[({p-1-m})(p-\gamma)]/({q-p+1})+p>p$, then
$$
u(x)=Ax^{-\alpha} \,\,\, \mbox{for} \,\,\, x>0,
$$
where $\alpha={(p-\gamma)}/{(q-p+1)}$ and $A >0$ is the unique solution of
\begin{equation}\label{u2}
QA^{q-m}-\alpha^{p-1}(1+\alpha)(p-1)A^{p-m-1}-R=0.
\end{equation}
\end{lemma}

Finally, the last lemma studies the behavior of the solution for a class of problem in the half space $D=\{x\in\mathbb{R}^N;\ x_1>0\}$.

\begin{lemma}\label{lemax1}
Let $p \in [m+1,q+1)$ and $\gamma\leq0$ be a real number. If $Q,R>0$ are real constants and $u\in C^1(D)$  is a solution of the problem
\begin{equation}\label{x1}
\left\{
\begin{array}{l}
-\Delta_pu=Rx_1^{-\eta}u^m-Qx_1^{-\gamma}u^q\,\, \mbox{in}\ D, \\
u>0~in ~D;~u=+\infty\ \ \mbox{on}\ \partial D,
\end{array}
\right.
\end{equation}
then
$$
u(x)=Ax_1^{-\alpha},~x\in D,
$$
where $\alpha$, $\eta$ and $A$ were obtained in Lemma $\ref{lemau1}$.
\end{lemma}

\noindent {\bf Proof of Theorem \ref{TP2}-Conclusion:}

Next, we will divide into two parts our proof. The first part is related to behavior near to boundary, while the second one is associated with the uniqueness.

\vspace{0.5 cm}

\noindent {\bf Part 1: Behavior near to boundary.} \\
\mbox{}\hspace{0.5 cm} Consider $x_0\in\partial\Omega$. We can assume that $x_0=0$ and $\nu(x_0)=-e_1$,
where $\nu(x_0)$ stands for the exterior normal derivative at $x_0$ and $e_1$ is the first vector of canonical basis of $\mathbb{R}^N$. Take
$x_n\subset\Omega$ such that $x_n\to x_0=0$ and denote by $\xi_n=x_n-t_n e_1$, where $t_n>0$ is such that $\xi_n\in\partial\Omega$. Now,
fixing $z_n=\xi_n-t_n\nu(\xi_n)$, we have that $d(z_n)=t_n$,
where $d_n:=d(z_n)=\inf\{\vert z_n - \xi \vert ~/~\xi \in \partial\Omega\}= \vert z_n - \xi_n\vert,~n \in \mathbb{N}\} $.

Now, fixing  $\alpha_n=\alpha(z_n)$ and
$$
v_n(y)=d_n^{\alpha_n}u(\xi_n+d_ny),~y\in\Omega_n=\{y\in\mathbb{R}^N,~
\xi_n+d_ny\in U_\delta\},
$$
where $U_\delta$ is a neighborhood of $\partial\Omega$ given in Lemma \ref{lemaU},
 it follows that
$$
|\nabla v_n(y)|^{p-2}\nabla v_n(y)=d_n^{(\alpha_n+1)(p-1)}|\nabla u(\xi_n+d_ny)|^{p-2}\nabla u(\xi_n+d_ny),~y\in\Omega_n.
$$
By change variable $z=\xi_n+d_ny$, we have that $y\in\Omega_n\Leftrightarrow z\in U_\delta$, and so,
 \begin{eqnarray}
\label{43}
\begin{array}{l}
 \displaystyle\int_{\Omega_n}|\nabla v_n(y)|^{p-2}\nabla v_n(y)\nabla\phi(y)dy=d_n^{\alpha_n(p-1)+p-N}\int_{U_\delta}|\nabla u(z)|^{p-2}
\nabla u(z)\nabla\phi\Big(\frac{z-\xi_n}{d_n}\Big)dz \\
\\
=\displaystyle d_n^{\alpha_n(p-1)+p-N}\int_{U_\delta}[\lambda a(z)g(u(z))-b(z_n)f(u(z))]\phi\Big(\frac{z-\xi_n}{d_n}\Big)dz \\
\\
 =\displaystyle d_n^{-N}\int_{U_\delta}[\lambda d_n^{\eta(z_n)}a(\xi_n+d_ny)d_n^{m\alpha_n}g(u(\xi_n+d_ny))\\
 \\
  \displaystyle -d_n^{\gamma(z_n)}b(\xi_n+d_ny)d_n^{q\alpha_n}f(u(\xi_n+d_ny))]\phi\Big(\frac{z-\xi_n}{d_n}\Big)dz \\
  \\
=\displaystyle d_n^{-N}\int_{U_\delta}[\lambda d_n^{\eta(z_n)-\eta(\xi_n+d_ny)}d_n^{\eta(\xi_n+d_ny)}a(\xi_n+d_ny)
d_n^{m\alpha_n}g(u(\xi_n+d_ny))\\
\\
\displaystyle -d_n^{\gamma(z_n)-\gamma(\xi_n+d_ny)}d_n^{\gamma(\xi_n+d_ny)}b(\xi_n+d_ny)d_n^{q\alpha_n}
f(u(\xi_n+d_ny))]\phi\Big(\frac{z-\xi_n}{d_n}\Big)dz\\
\\
=\displaystyle \int_{\Omega_n}[\lambda d_n^{\eta(z_n)-\eta(\xi_n+d_ny)}d_n^{\eta(\xi_n+d_ny)}a(\xi_n+d_ny)
d_n^{m\alpha_n}g(u(\xi_n+d_ny))\\
\\
\displaystyle -d_n^{\gamma(z_n)-\gamma(\xi_n+d_ny)}d_n^{\gamma(\xi_n+d_ny)}b(\xi_n+d_ny)d_n^{q\alpha_n}
f(u(\xi_n+d_ny))]\phi(y)dy],
 \end{array}
\end{eqnarray}
for each $\phi\in C_0^\infty(\Omega_n)$. Since $\Omega_n\to D$ when $n\to+\infty$, where
$D=\{y\in\mathbb{R}^N,~y_1>0\}$, we obtain for each compact set $K\subset\subset D$ given, that there exists an $n_0 \in \mathbb{N}$ such that $K\subset\subset\Omega_n$ and $\xi_n+d_ny\in U_\delta$ for all $y\in K$ and $n>n_0$, where $U_\delta$ is given at Lemma \ref{lemaU}. Thus, from the regularity of distance function, see for instance \cite[Lemma 14.16]{Trudinger},
\begin{equation}
\label{432}
 \displaystyle\frac{d(\xi_n+d_ny)}{d(z_n)}=\frac{d(\xi_n+d_ny)-d(\xi_n)}{d(z_n)}=\frac{\langle \nabla d(\varsigma_n),
d_ny\rangle}{d_n}\to \langle \nabla d(0), y\rangle=\langle e_1,y\rangle=y_1,
\end{equation}
uniformly in $y \in K$, for some $\varsigma_n$ between $\xi_n+d_ny$ and $\xi_n$.

Thereby, $(a)$ combined with the above convergences gives
\begin{eqnarray}
\label{44}
\begin{array}{l}
  \displaystyle d_n^{\eta(\xi_n+d_ny)}a(\xi_n+d_ny) \\
  \\
  =\displaystyle \Big(\frac{d(z_n)}{d(\xi_n+d_ny)}\Big)^{\eta(\xi_n+d_ny)}d(\xi_n+d_ny)^{\eta(\xi_n+d_ny)}
a(\xi_n+d_ny)\to y_1^{-\eta(0)}R(0),~y \in K.
\end{array}
\end{eqnarray}
With the same type of arguments, combining $(b)$ with the  convergence at (\ref{432}), we see that
\begin{eqnarray}
\label{45}
\begin{array}{l}
  \displaystyle d_n^{\gamma(\xi_n+d_ny)}b(\xi_n+d_ny) \\
  \\
  =\displaystyle \Big(\frac{d(z_n)}{d(\xi_n+d_ny)}\Big)^{\gamma(\xi_n+d_ny)}d(\xi_n+d_ny)^{\gamma(\xi_n+d_ny)}
b(\xi_n+d_ny)\to y_1^{-\gamma(0)}Q(0),~y \in K.
\end{array}
\end{eqnarray}
To complete our analysis of convergence, from Lemma \ref{lemaU},
$$
\begin{array}{lll}
 \displaystyle v_n(y)&\leq& c_2 d_n^{\alpha_n} d(\xi_n+d_ny)^{-\alpha(\xi_n+d_ny)}\\
&=&\displaystyle c_2\Big(\frac{d_n}{d(\xi_n+d_ny)}\Big)^{\alpha(\xi_n+d_ny)}
d_n^{\alpha_n-\alpha(\xi_n+d_ny)},~y \in K
\end{array}
$$
and
$$
\begin{array}{lll}
 \displaystyle v_n(y)&\geq& c_1d_n^{\alpha_n}d(\xi_n+d_ny)^{-\alpha(\xi_n+d_ny)}\\
&=& \displaystyle c_1\Big(\frac{d_n}{d(\xi_n+d_ny)}\Big)^{\alpha(\xi_n+d_ny)}
d_n^{\alpha_n-\alpha(\xi_n+d_ny)},~y \in K.
\end{array}
$$
Furthermore, from $(b)$,
$$\vert \ln d_n^{\alpha_n-\alpha(\xi_n+d_ny)}\vert=\vert(\alpha(z_n)-\alpha(\xi_n+d_ny))\ln d_n\vert\leq \hat{c}d_n^\mu\vert\ln d_n\vert\to0$$
uniformly in $y \in K$, for some $\hat{c}>0$, implying that
 \begin{equation}
 \label{4321}
 d_n^{\alpha_n-\alpha(\xi_n+d_ny)}\to 1~\mbox{uniformly in } y \in K.
\end{equation}

Gathering  (\ref{432}), (\ref{4321}), regularity of distance function with the fact that $(v_n)$ is uniformly  bounded on compacts set in $D$, we derive that there is  a function $v$ such that  $v_n(y) \to v(y)$  and $c_1y_1^{-\alpha_0}\leq v(y)\leq c_2y_1^{-\alpha_0}$, for each $y \in D$.

After that, by  $(g_0)^{\prime}$,
\begin{eqnarray}
\label{46}
\begin{array}{lll}
 d_n^{m\alpha_n}g(u(\xi_n+d_ny))&=&d_n^{m\alpha_n}u^m(\xi_n+d_ny)u^{-m}(\xi_n+d_ny)g(u(\xi_n+d_ny))\\
 \\
&=&v_n^m(y)u^{-m}(\xi_n+d_ny)g(u(\xi_n+d_ny))\\
\\
&\to& g_{\infty}v(y)^{m},~y \in D
\end{array}
\end{eqnarray}
and by $(f_1)^{\prime}$,
\begin{eqnarray}
\label{47}
\begin{array}{lll}
 d_n^{q\alpha_n}f(u(\xi_n+d_ny))&=&d_n^{q\alpha_n}u^q(\xi_n+d_ny)u^{-q}(\xi_n+d_ny)f(u(\xi_n+d_ny))\\
 \\
&=&v_n^q(y)u^{-q}(\xi_n+d_ny)f(u(\xi_n+d_ny))\\
\\
&\to& f_{\infty}v(y)^q,~y \in D.
\end{array}
\end{eqnarray}

Finally, as $\eta,\gamma \in C^{\mu}(\overline{\Omega})$, for some $0<\mu <1$,  repeating the same arguments used in the proof  of (\ref{4321}), we deduce that
\begin{equation}
\label{431}
 d_n^{\eta(z_n)-\eta(\xi_n+d_ny)},~d_n^{\gamma(z_n)-\gamma(\xi_n+d_ny)}\to 1 ~\mbox{with}~ n\to+\infty, \,\,\, \mbox{for each} \,\,\, y \in D.
\end{equation}

Now, given $\phi \in C^{\infty}_0(D)$ and recalling that $\Omega_n \to D$, we have $\overline{supp \phi} \subset \Omega_n  $ for $n$
large enough. Thereby, passing the limits in (\ref{43}), and using (\ref{44}), (\ref{45}), (\ref{46}) and (\ref{47}),   we conclude that
 $v_n\to v$ in $C^1_{loc}(D)$ and $v$ is a solution of the problem
$$
\left\{
\begin{array}{l}
-\Delta_pu=\lambda g_{\infty} R(0)y_1^{-\eta(0)}u^m-f_{\infty}Q(0)y_1^{-\gamma(0)}u^q \,\,\, \mbox{in}\ D, \\
\\
c_1y_1^{-\alpha(0)}\leq u\leq c_2y_1^{-\alpha(0)}~\, \mbox{in}\ D,~~u=+\infty\ \ \mbox{on}\ \partial D.
\end{array}
\right.
$$
Hence, taking $Q=f_\infty Q(0)$, $R=\lambda g_\infty R(0)$ and $\alpha=\alpha(0)$, the Lemma \ref{lemax1} gives
 $$v(y)=Ay_1^{-\alpha(0)},~y \in D,$$
 where $A=A(0)>0$ is the unique solution of
$$
f_{\infty} Q(0) A^{q-m}-(p-1)\alpha(0)^{p-1}(1+\alpha(0))A^{p-m-1}-\lambda g_{\infty} R(0)=0.
$$

Now, taking $y=e_1$ and using the definition of $v_n$, we obtain that
\begin{equation}\label{fin}
\lim_{n\to+\infty}d_n^{\alpha_n}u(x_n)=A.
\end{equation}
To complete our proof,  we will use the following limits
\begin{equation}\label{fin1}
d_n^{-\alpha_n}d_n^{\alpha(x_n)},~d_n^{-\alpha(x_n)}d(x_n)^{\alpha(x_n)} \to 1.
\end{equation}
We prove the above limits following the same arguments like those used to prove (\ref{4321}), because $\alpha \in C^{\mu}(\overline{\Omega})$ for some $0<\mu <1$. Moreover, another important limit involving the function $d$ is
$
{d(x_n)}/{d(z_n)}\to 1,
$
 whose proof follows similar arguments like those used to prove (\ref{432}).

Therefore, from (\ref{fin}) and (\ref{fin1})
$$
\lim_{n\to+\infty}d(x_n)^{\alpha(x_n)}u(x_n)=\lim_{n\to+\infty}[d_n^{-\alpha_n}d_n^{\alpha(x_n)}][d_n^{-\alpha(x_n)}
d(x_n)^{\alpha(x_n)}][d_n^{\alpha_n}u(x_n)]=A.
$$

\vspace{0.5 cm}

\noindent {\bf Part 2: Uniqueness.} \\
Let $u, v$ be two solutions of $(P_\lambda)$. By the above information, $$
\lim_{x\to x_0}\frac{u(x)}{v(x)}=1, \,\,\, \mbox{for each} \,\,\, x_0\in\partial\Omega.
$$
Combining the last limit with the compactness of $\partial\Omega$, for each $\epsilon>0$,  there exists $\delta>0$ such that
\begin{equation}\label{desi}
(1-\epsilon)v(x)\leq u(x)\leq (1+\epsilon)v(x),~x\in U_{\delta}.
\end{equation}
Moreover, using that ${f(t)}/{t^{p-1}}$ is nondecreasing and
${g(t)}/{t^{p-1}}$ is nonincreasing in the interval $(0,+\infty)$, we deduce that $(1-\epsilon)v$ and $(1+\epsilon)v$ are sub and super solutions of the problem
\begin{equation}\label{pepsilon}
\left\{
\begin{array}{l}
 -\Delta_pw=\lambda a(x)g(w)-b(x)f(w)\ \mbox{in}\ U^\delta,\\
w=u\ \mbox{on}\ \partial U^\delta,
\end{array}
\right.
\end{equation}
where $U^\delta=\{x\in \Omega, d(x)>\delta\}$. Since $u$ is a solution of (\ref{pepsilon}) as well, it follows from Lemma \ref{L2},
$$
(1-\epsilon)v(x)\leq u(x)\leq (1+\epsilon)v(x),~x\in  U^\delta.
$$
Now, combining the last inequality with (\ref{desi}), we are led to
$$
(1-\epsilon)v(x)\leq u(x)\leq (1+\epsilon)v(x),~x \in \Omega.
$$ Taking $\epsilon\to0$, we obtain $u=v$ in $\Omega$, finishing the proof.

 \section{Proof of Lemmas}

\noindent \textbf{Proof of Lemma \ref{lemaU}}

\noindent\begin{proof}: Given $x \in U_{\delta}$, where $\delta>0$ is given by hypotheses $(a)$ and $(b)$, define the function 
$$
v(y)=d(x)^{\alpha(x)}u(x+d(x)y),\ y\in B_{{1}/{2}}(0).
$$
As $u \in C^1(\Omega)$ is a solution of ${(P)_{\lambda}}$, for each $\varphi\in C_0^\infty(B_{{1}/{2}}(0))$, the change of variable $z=x+d(x)y$ leads to
$$
\begin{array}{l}
 \displaystyle \int_{B_{{1}/{2}}(0)}|\nabla v|^{p-2}\nabla v\nabla\varphi(y)dy=d(x)^{\alpha(x)(p-1)+p-N}\int_{B_{{d(x)}/{2}}(x)}
|\nabla u(z)|^{p-2}\nabla
u(z)\nabla\varphi\Big(\frac{1}{d(x)}(z-x)\Big)dz \\
\\
=\displaystyle d(x)^{\alpha(x)(p-1)+p-N}\int_{B_{{d(x)}/{2}}(x)}[\lambda a(z)g(u(z))-b(z)f(u(z))]\varphi\Big(\frac{1}{d(x)}(z-x)\Big)dz
\end{array}
$$
Now, gathering the compactness of $\partial \Omega$,  $(f_1)^{\prime}$ and $(g_0)^{\prime}$, we derive that
\begin{equation}
\label{51}\begin{array}{l}
 \displaystyle \int_{B_{{1}/{2}}(0)}|\nabla v|^{p-2}\nabla v\nabla\varphi(y)dy \leq
d(x)^{-N}\int_{B_{{d(x)}/{2}}(x)}d(x)^{\alpha(x)(p-1)+p}\Big[\lambda a(x+d(x)y)D_2u^m(x+d(x)y)\\
 \\
\displaystyle  -b(x+d(x)y)D_1'u^q(x+d(x)y)\Big]
\varphi\Big(\frac{1}{d(x)}(z-x)\Big)dz \\
\\
\displaystyle =d(x)^{-N}\int_{B_{{d(x)}/{2}}(x)}\Big[\lambda D_2a(x+d(x)y)d(x)^{\alpha(x)(p-1)+p}d(x)^{-m\alpha(x)}v^m(y)\\
 \\
\displaystyle -D_1'b(x+d(x)y)d(x)^{\alpha(x)(p-1)+p}d(x)^{-q\alpha(x)}v^q(y)\Big]\varphi\Big(\frac{1}{d(x)}(z-x)\Big)dz\\
 \\
\displaystyle =d(x)^{-N}\int_{B_{{d(x)}/{2}}(x)}\Big[\lambda D_2a(x+d(x)y)d(x)^{\eta(x)}v^m(y)
-D_1'b(x+d(x)y)d(x)^{\gamma(x)}v^q(y)\Big]\varphi\Big(\frac{1}{d(x)}(z-x)\Big)dz,
\end{array}
\end{equation}
for all $\varphi\in C_0^\infty(B_{{1}/{2}}(0))$ with $\varphi\geq0$ and for all
$x \in U_{\delta},$
where $\delta>0$ is such that
$$
g(u(x))\leq D_2u(x)^m \,\,\, \mbox{and} \,\,\, f(u(x))\geq D_1'u(x)^q \,\,\, \mbox{for all} \,\,\,\, x \in U_{\delta},
$$
for some real constants $D_2,D_1'>0$. Here, we have used that $u(x) \to \infty$ as $d(x) \to 0$.

Moreover, using the inequality below
$$
d(x)/2\leq d(x+d(x)y)\leq {3}d(x)/2 \,\,\, \forall x \in U_{\delta}
$$
together with $(a)$ and $(b)$, we get
$$
\begin{array}{l}
  b(x+d(x)y)\geq \tilde{C} d(x+d(x)y)^{-\gamma(x+d(x)y)}\geq C d(x)^{-\gamma(x+d(x)y)},~x \in U_{\delta}
\end{array}
$$
and
$$
\begin{array}{l}
  a(x+d(x)y)\leq \tilde{D}d(x+d(x)y)^{-\eta(x+d(x)y)} \leq D d(x)^{-\eta(x+d(x)y)},~x \in U_{\delta}
\end{array}
$$
for suitable $\delta>0$ and some positive constants $\tilde{C}, \tilde{D}, C$ and $D$. Therefore,
$$
\begin{array}{l}
\lambda D_2a(x+d(x)y)d(x)^{\eta(x)}v^m(y)-D_1'b(x+d(x)y)d(x)^{\gamma(x)}v^q(y) \\
 \\
  \leq\lambda D_3 d(x)^{\eta(x)-\eta(x+d(x)y)}v^m(y)-D_1d(x)^
{\gamma(x)-\gamma(x+d(x)y)}v^q(y),
 \end{array}
$$
for $x\in  U_{\delta},~y \in B_{1/2}(0)$ and $D_1, D_3>0$.

Now, substituting this inequality in (\ref{51}) and returning to the variable $y \in B_{1/2}(0)$, we obtain
$$
\int_{B_{{1}/{2}}(0)}|\nabla v|^{p-2}\nabla v\nabla\varphi(y)dy\leq\int_{B_{{1}/{2}}(0)}[\lambda D_3 d(x)^{\eta(x)-\eta(x+d(x)y)}
v^m(y)-D_1 d(x)^{\gamma(x)-\gamma(x+d(x)y)}v^q(y)]\varphi(y)dy.
$$
Taking the limit $x \to \partial \Omega$, or equivalently $d(x) \to 0 $, we find
\begin{equation}
\label{52}
\int_{B_{{1}/{2}}(0)}|\nabla v|^{p-2}\nabla v\nabla\varphi(y)dy\leq \int_{B_{{1}/{2}}(0)}
[\lambda D_3 v^m(y)-D_1v^q(y)]\varphi(y)dy,
\end{equation}
because
$$
d(x)^{\eta(x)-\eta(x+d(x)y)}, d(x)^{\gamma(x)-\gamma(x+d(x)y)} \to 1 \,\,\, \mbox{as} \,\,\,\, d(x) \to 0.
$$

On the other hand, from Theorem \ref{T1}, there exists $U \in C^1(B_{{1}/{2}}(0))$ satisfying
\begin{equation}\label{U}
\left\{
\begin{array}{l}
-\Delta_p{U}=\lambda D_3 U^m- D_1 U^q \,\,\, \mbox{in} \,\, B_{{1}/{2}}(0), \\
U>0 \,\,\, \mbox{in} \,\,\, B_{{1}/{2}}(0),\ U=+\infty \,\,\,
\mbox{on} \,\,\, \partial B_{{1}/{2}}(0).
\end{array}
\right.
\end{equation}
Then, by  Lemma \ref{L2},
$$
v(y)\leq U(y)\ \mbox{in}\ B_{{1}/{2}}(0),
$$
that is,
$$
d(x)^{\alpha(x)}u(x+d(x)y)\leq U(y)\ \mbox{for all }\ y \in B_{{1}/{2}}(0)~\mbox{and}~x \in U_{\delta},
$$
showing that
\begin{equation} \label{EST1}
u(x)\leq U(0)d(x)^{-\alpha(x)} \,\,\, \mbox{for} \,\,\, x \in U_{\delta}.
\end{equation}

Now, let us prove the other inequality. Denote by $\bar{x} \in \partial\Omega$ the point that carries out the distance of $x$ on $\partial\Omega$, and fix
$z_x=\bar{x}+d(x)\nu(\bar{x})$, where $\nu(\bar{x})$ is the exterior unity normal vector  to the $\partial\Omega$ at $\bar{x}$. Since $\partial\Omega$ is smooth, we have that $z_x \in \Omega^c$ for $x \in U_{{\delta}/2}$ for some $\delta>0$. This way, we can define
$$
w(y):=d(x)^{\alpha(x)}u(z_x+d(x)y),~y\in Q_x=\{y\in A~/~ z_x+d(x)y\in U_{\delta}\},
$$
where $A=\{y\in\mathbb{R}^N~/~1<|y|<3\}$.

From the hypotheses  $(b)$, we can fix $\delta>0$ small enough, such that
\begin{equation}\label{53}
\begin{array}{c}
  b(z_x+d(x)y)\leq C_1 d(z_x+d(x)y)^{-\gamma(z_x+d(x)y)} ,
\end{array}
\end{equation}
 and
\begin{equation}\label{54}
1/2 \leq d(x)^{\eta(x)-\eta(z_x+d(x)y)},d(x)^{\gamma(x)-\gamma(z_x+d(x)y)}\leq 3/2
\end{equation}
for all $x \in  U_{{\delta}/2}$ and some $C_1>0$. In the sequel, by using  $(f_1)^{\prime}$  and the fact that $u(x) \to \infty$ as $\vert x \vert \to \infty$, we  can also fix  $C_2>0$ verifying
\begin{equation}\label{56}
f(u(x))\leq C_2 u(x)^q \,\,\, \forall x \in  U_{{\delta}/2}.
\end{equation}

Thus, given $\varphi\in C_0^\infty(Q_x)$ with $\varphi\geq0$, (\ref{53}) together with (\ref{56}) and the positivity of $a$ on $U_{\delta}$, yield
\begin{equation}\label{55}
\begin{array}{lcl}
  \displaystyle\int_{Q_x}|\nabla w|^{p-2}\nabla w\nabla\varphi dy & = & \displaystyle d(x)^{(\alpha(x)+1)(p-1)}\int_{Q_x}
|\nabla u(z_x+d(x)y)|^{p-2}\nabla u(z_x+d(x)y)\nabla\varphi dy \\
\\
   & \geq & \displaystyle -\int_{Q_x}C_1 C_2 d(x)^{\gamma(x)-\gamma(z_x+d(x)y)}w^q(y)
\varphi dy.
\end{array}
\end{equation}
From  (\ref{54}) and (\ref{55}),
\begin{equation}\label{57}
\int_{Q_x}|\nabla w|^{p-2}\nabla w\nabla\varphi(y)dy\geq -\int_{Q_x}
C_3w^q(y)\varphi(y)dy
\end{equation}
for $x \in  U_{{\delta}/2}$  and some $C_3>0$.

On the other hand, set $\hat{Z} \in C^1 (1,3)$ denotes the positive solution of
\begin{equation}\label{135}
\left\{
\begin{array}{l}
-(r^{N-1}|Z'|^{p-2}Z')'=- C_3 r^{N-1} Z^q\,\, \mbox{in} \,\, (1,3), \\
Z>0~ \mbox{in}~(1,3); ~Z(1)=K, \, \, \, Z(3)=0,
\end{array}
\right.
\end{equation}
then $Z(y)=\hat{Z}(\vert y \vert) \in C^1(A)$ is a radially-symmetric  solution of the problem
\begin{equation}\label{ZZ}
\left\{
\begin{array}{l}
-\Delta_p{Z}=-C_3 Z^q\,\, \mbox{in} \,\, A, \\
Z>0~ \mbox{in}~(1,3); ~Z(1)=K, \, \, \, Z(3)=0.
\end{array}
\right.
\end{equation}
Since $Q_x \subset A$, it follows that $Z(y) \leq w(y)$, $y \in \partial Q_x$. So, the inequality  (\ref{57}) combined with (\ref{ZZ}) and Lemma \ref{L2} gives
$$
d(x)^{\alpha(x)}u(z_x+d(x)y)=w(y)\geq Z(y)\ \mbox{in}\
Q_x,~\mbox{for all}~x \in U_{{\delta}/2},
$$
that is, taking $y=-2\nu(\bar{x})$ and remembering that $x=z_x - 2d(x)\nu(\bar{x})$, we obtain
\begin{equation} \label{EST2}
u(x)\geq Z(-2\nu(\bar{x}))d(x)^{-\alpha(x)}=\hat{Z}(2)d(x)^{-\alpha(x)},~x \in U_{{\delta}/2}.
\end{equation}
Now,  the lemma follows gathering (\ref{EST1}) and (\ref{EST2}) by considering the smallest $\delta>0$ that we have considered in this proof.
\end{proof}
 \fim
\medskip

\vspace{0.5 cm}

\noindent \textbf{Proof of Lemma \ref{lemau1}}
\medskip

The proof of Lemma \ref{lemau1} is based upon ideas found in  \cite{Goncalvessantos}. Here, we are able to prove that the solutions of problem (\ref{u1}) are of the form $u(x)=Ax^{-\alpha}$, with $A$ verifying (\ref{u2}), by using a result of \cite{Goncalvessantos} instead of the Poincar\'e-Bendixon's Theorem as used in  \cite{Garcia-Melian}. More exactly, the results that we will use has the following statement: \\

Given positive numbers $T_1, T_2$ and $h$, we let $X:=\{w\in C^1([T_1,T_2])~/~ w\geq h\}$ and the continuous function
 $H: [T_1,T_2]\to \mathbb{R}$ defined by
$$H(s):=s^{N-1}[|(w_2^{{1}/{p}})'|^{p-2}(w_2^{{1}/{p}})'w_2^{({1-p})/{p}}-
|(w_1^{{1}/{p}})'|^{p-2}(w_1^{{1}/{p}})'w_1^{({1-p})/{p}}](w_1-w_2)(s)$$
for  $w_1,\ w_2\in X$ given. In \cite{Goncalvessantos}, it was proved the following result

\begin{lemma}\label{lemaapendice}
Assume that  $w_1,w_2\in X$, then
$$H(U)-H(S)\leq
\int_S^U\Big[\frac{(r^{N-1}|(w_2^{{1}/{p}})'|^{p-2}(w_2^{{1}/{p}})')'}{w_2^{({p-1})/{p}}}-
\frac{(r^{N-1}|(w_1^{{1}/{p}})'|^{p-2}(w_1^{{1}/{p}})')'}{w_1^{({p-1})/{p}}}\Big](w_1-w_2)dr$$
for all $U,S$ such that $T_1\leq S\leq U\leq T_2$ hold.
\end{lemma}

\noindent\begin{proof} $of~ Lemma ~\ref{lemau1}$:
It is easy to check that $u_0(x):=Ax^{-\alpha}$, $x>0$ is a solution of (\ref{u1}), where $A>0$ is the unique solution of (\ref{u2}). In the sequel, we will show that $u_0$ is a maximal solution for (\ref{u1}). To see why, our first step is to show that if $u \in C^1(0,\infty)$ is a solution of (\ref{u1}), then
\begin{equation} \label{EST3}
u(x)\leq cx^{-\alpha} \,\,\, \forall x>0, \,\,\,
\end{equation}
for some positive constant $c$. Fixed $x>0$, define $v(y)=x^\alpha u(x+xy)$ for  $|y|<{1}/{2}$, and note that $v$ satisfies
\begin{equation}\label{U11}
\left\{
\begin{array}{l}
-(|v'|^{p-2}v')'=R(1+y)^{-\eta}v^m - Q(1+y)^{-\gamma}v^q,~|y|<{1}/{2}, \\
v>0~~\mbox{in}~~|y|<{1}/{2},~~ v({1}/{2})=x^\alpha u(3x/2)~\mbox{and}~v(-{1}/{2})=x^\alpha u(x/2).
\end{array}
\right.
\end{equation}
On the other hand, from Theorem \ref{T1}, there exists $U \in C^1(-1/2,1/2)$ satisfying
\begin{equation}\label{U1}
\left\{
\begin{array}{l}
-(|U'|^{p-2}U')'=R(1+y)^{-\eta}U^m-Q(1+y)^{-\gamma}U^q, \,\,\, |y|<{1}/{2}, \\
U>0~~\mbox{in}~~|y|<{1}/{2},~~ U({1}/{2})=U(-{1}/{2})=+\infty.
\end{array}
\right.
\end{equation}
Combining (\ref{U11}) with (\ref{U1}) and Lemma \ref{L2}, we deduce that
$$
v(y)\leq U(y) \,\,\, \mbox{for} \,\,\,  |y|<{1}/{2}.
$$
Taking $y=0$, we see that
$$
u(x)\leq U(0)x^{-\alpha}, \,\, x>0 \,\,\, (c=U(0)>0),
$$
proving (\ref{EST3}).

After the previous study, we are able to prove that
\begin{equation} \label{EST4}
u(x) \leq u_0(x) \,\,\, \forall x>0.
\end{equation}
To this end, we assume that there exists $\tau_0>0$ such that $u\leq\zeta_{\tau_0}$ does not hold in $(\tau_0,\infty)$, where  $\zeta_{\tau}(x):=u_0(x-\tau)$ for $x>\tau$ for each $\tau>0$ given. Thereby, there exist $t_0,s_0 \in [\tau_0, \infty]$ such that $u(t_0)=\zeta_{\tau_0}(t_0)$, $u(s_0)=\zeta_{\tau_0}(s_0)$, if $s_0<\infty$ and $u(x)>\zeta_{\tau_0}(x)$ in $(t_0,s_0)$.

A straightforward computation gives that $\zeta_{\tau_0}$ satisfies
\begin{equation}\label{59}
\begin{array}{l}
-(|\zeta_{\tau_0}'(x)|^{p-2}\zeta_{\tau_0}'(x))'\geq Rx^{-\eta}\zeta_{\tau_0}^m - Qx^{-\gamma}\zeta_{\tau_0}^q \,\mbox{in}\, (t_0,s_0).
\end{array}
\end{equation}
Putting $N=1$, $w_1^{\frac{1}{p}}=u$ and $w_2^{\frac{1}{p}}=\zeta$ into Lemma \ref{lemaapendice}, (\ref{59}) together with the fact that $u$ is a solution of (\ref{u1}) yields
$$
\begin{array}{lcl}
 H(s_2)-H(s_1)&\leq&\displaystyle \int_{s_1}^{s_2}\Big[\frac{(|\zeta_{\tau_0}'|^{p-2}\zeta_{\tau_0}')'}{\zeta_{\tau_0}^{p-1}}-\frac{(|u'|^{p-2}u')'}{u^{p-1}}\Big](u^p-\zeta_{\tau_0}^p)dx\\
 \\
 &\leq&\displaystyle \int_{s_1}^{s_2}\Big[\frac{Qx^{-\gamma}\zeta_{\tau_0}^q-Rx^{-\eta}\zeta_{\tau_0}^m}{\zeta_{\tau_0}^{p-1}}-
\frac{Qx^{-\gamma}u^q-Rx^{-\eta}u^m}{u^{p-1}}\Big](u^p-\zeta_{\tau_0}^p)dx\\
\\
&=&\displaystyle \int_{s_1}^{s_2}[Qx^{-\gamma}(\zeta_{\tau_0}^{q-p+1}-u^{q-p+1})+Rx^{-\eta}(u^{m-p+1}-\zeta_{\tau_0}^{m-p+1})](u^p-\zeta_{\tau_0}^p)dx<0,
\end{array}
$$
for all $t_0\leq s_1<s_2< s_0$, where
\begin{equation}\label{511}
H(x)=\big[|\zeta_{\tau_0}'|^{p-2}\zeta_{\tau_0}'\zeta_{\tau_0}^{(1-p)}-|u'|^{p-2}u'u^{(1-p)}\big](u^p(x)-\zeta_{\tau_0}^p(x)),~x \in (t_0,s_0).
\end{equation}
The above inequality implies that $H$ is decreasing in $(t_0,s_0)$. Thus, if $s_0<+\infty$, then $H(t_0)=H(s_0)=0$, that is impossible. If $s_0=+\infty$, then $\displaystyle \lim_{x\to+\infty}H(x)=H_{\infty} \in [-\infty, 0)$, because $H(t_0)=0$ and $H$ is decreasing.

Moreover, the definition of $\zeta_{\tau_0}$ and (\ref{EST3}) combine to give
$$
\lim_{x\to+\infty}|\zeta_{\tau_0}'|^{p-2}\zeta_{\tau_0}'\zeta_{\tau_0}^{(1-p)}(x)=\lim_{x\to+\infty}(u^p-\zeta_{\tau_0}^p)(x)=0.
$$
Then, by (\ref{511}) and $H_{\infty} \in [-\infty, 0)$,
$$
\lim_{x\to+\infty}|u'|^{p-2}u'u^{(1-p)}(x)=+\infty,
$$
showing that $u'>0$ for $x$ large enough, which is impossible, because $u(x)\buildrel
x \to \infty \over \longrightarrow 0$. Hence,
$$
u(x)\leq \zeta_{\tau}(x)  \,\,\, \forall x \in  (\tau,+\infty) \,\,\, \mbox{for all} \,\,\, \tau>0,
$$
implying that
$$
u(x)\leq \lim_{\tau\to 0}\zeta_{\tau}(x)= u_0(x) \,\,\, \forall x \in (0,+\infty),
$$
showing (\ref{EST4}), and thus, $u_0$ is a maximal solution for (\ref{u1}).

To complete the proof of Lemma  \ref{lemau1}, we will show that $u_0$ is also a minimal solution for (\ref{u1}). In the sequel, we define $\xi_{\epsilon}(x)=u_0(x+\epsilon)$ in $(0,+\infty)$ for each $\epsilon>0$ and we use a similar argument to conclude that for each $\epsilon >0$ the inequality below holds
$$
u(x)\geq \xi_{\epsilon}(x) \,\,\,\, \forall x \in (0,+\infty).
$$
The above estimate leads to
$$
u(x)\geq \lim_{\epsilon \to 0}\xi_{\epsilon}(x)=u_0(x) \,\,\, \forall x \in(0,+\infty),
$$
from where it follows that $u_0$ is a minimal solution.  Since $u_0$ is at the same time a maximal and minimal solution,
we can conclude that $u(x)=Ax^{-\alpha},~x>0$ is the unique solution of (\ref{u1}), finishing the proof of the lemma. \end{proof} \fim
\medskip

\noindent \textbf{Proof of Lemma \ref{lemax1}}
\medskip

\noindent\begin{proof} In this proof, our first step is to show that 
$$
u_0(x)=u_0(x_1,x_2, \cdots, x_n)=Ax_1^{-\alpha}
$$ 
is a solution of (\ref{x1}), where $A>0$ is the unique solution of (\ref{u2}). Below, we prove that (\ref{x1}) admits a minimal and a maximal solutions depending on just $x_1$. In fact, we will begin showing the existence of the maximal solution, which we will be denoted by $u_{max}$.

To do this, let $\{D_k\}$ be a sequence of smooth bounded domains $D_k\subset\subset D_{k+1}$  such that $D=\cup_{k=1}^\infty D_k$. Related to   $\{D_k\}$, we consider the problem
\begin{equation}\label{512}
\left\{
\begin{array}{l}
- \Delta_pu=Rx_1^{-\eta}u^m-Qx_1^{-\gamma}u^q \,\,\, \mbox{in}\ D_k, \\
u>0~\mbox{in}~D_k,~~u=+\infty\ \ \mbox{on}\ \partial D_k.
\end{array}
\right.
\end{equation}
By Theorem \ref{T1},  there exists a solution $u_k \in C^1(D_k)$ of (\ref{512}) satisfying
$$
u_0(x) \leq u_{k+1}(x)\leq u_k(x), \,\,\, x \in D_k.
$$
The above inequalities follow from Lemma \ref{L2}. Thus, there is $w \in C^{1}(\Omega)$ such that $u_k \to w$ in $C^{1}_{loc}(\Omega)$, $w$ is a solution of (\ref{x1}) and
$$
w(x) \geq u_0(x) \,\,\, \forall x \in D.
$$

Let $v\in C^1(\Omega)$ be another solution of (\ref{x1}). By Lemma \ref{L2}, $v\leq u_k$ in $D_k$ for all $k$. Then,
$v\leq w$ in $D$, showing that $w$ is a maximal solution for (\ref{x1}). In the sequel, we denote by $u_{max}$ the function ${w}$ and set
$$
\tilde{w}(x)=u_{max}(x_1,x'+t), \,\,\, \mbox{for} \,\,\, x_1>0~\mbox{and}~x'\in\mathbb{R}^{N-1},
$$
for each $t\in\mathbb{R}^{N-1}$ given.

Since, $\tilde{w}$ is a solution of (\ref{x1}) as well, it follows that $\tilde{w} \leq u_{max}$ in $D$, or equivalently,
$$
u_{max}(x_1,x'+t)\leq u_{max}(x_1,x')\ \mbox{for each}\ x_1>0\ \mbox{and}\ t,x'\in\mathbb{R}^{N-1}
$$
given. So, it follows from the arbitrariness of $t\in\mathbb{R}^{N-1}$ and the above inequality, that
$$
u_{max}(x_1,x')=u_{max}(x_1,y') \,\,\, \forall x',y' \in \mathbb{R}^{N-1},
$$
showing that $u_{max}$ depends just on $x_1$. Thereby, $u_{max}$ is a solution of
problem (\ref{u1}), and by Lemma \ref{lemau1},
$$
u_{max}(x_1,x_2, \dots, x_n)=Ax_1^{-\alpha},~x_1>0 \,\,\, \mbox{and} \,\,\, (x_2, \dots,x_n) \in \mathbb{R}^{N-1}.
$$

To finish the proof, our next step is to prove the existence of a minimal solution for (\ref{x1}),  denoted by $u_{min}$, which will also depend on just $x_1$. To do this, taking $D_k'=B_{4k}(0)\cap D$, we have that
$$
D_k'\subset D_{k+1}', ~B_{3k}(0)\cap\partial D\subset\partial D_k'~\mbox{and}~D=\cup_{k=1}^\infty D_k'.
$$
From now on, for each $k \in \mathbb{N}$, we fix $\psi_k \in C^{\infty}(D_k')$ satisfying $0\leq\psi_k\leq1$ on $\partial D_k'$, $\psi_k=1$ on
$\partial D\cap B_k(0)$, $\psi_k=0$ in $\partial D_k'\setminus (B_{2k}(0)\cap\partial D)$ and $\psi_{k+1}\geq \psi_k$
on $\partial D_k'\cap\partial D_{k+1}'\cap\partial D$.

By a result found in \cite{Garcia-Melian}, there exists a unique solution $\underline{u}_{k,n} \in C^1(\overline{D_k'})$
of the problem
$$
\left\{
\begin{array}{l}
-\Delta_pu=-Qx_1^{-\gamma}u^q \,\,\, \mbox{in}\ D_k', \\
u>0~\mbox{in}~D_k',~~u=n\psi_k\ \ \mbox{on}\ \partial D_k',
\end{array}
\right.
$$
that is,  $\underline{u}_{k,n}$ is a sub solution of  the problem
\begin{equation}
\label{513}
\left\{
\begin{array}{l}
-\Delta_pu=Rx_1^{-\eta}u^m-Qx_1^{-\gamma}u^q \,\,\, \mbox{in}\ D_k^r , \\
u>0~\mbox{in}~D_k^r,~~u=\underline{u}_{k,n}\ \ \mbox{on}\ \partial D_k^r,
\end{array}
\right.
\end{equation}
where $D_k^r=B_{4k}(0)\cap \{x\in D; x_1>r\}\subset D_k'$ for each $r \in (0,({A}/{n})^{{1}/{\alpha}})$.

Since $u_0$ is a super solution of (\ref{513}) with $\underline{u}_{k,n} \leq u_0$ on $\partial D_k^r$, there exists a $v_{k+1,n}^r \in C^1(\overline{D_k^r})$ solution of the problem (\ref{513}) satisfying $\underline{u}_{k,n} \leq v_{k,n}^r \leq u_0$ in $D_k^r$. Then, after a diagonal process, there is  $v_{k,n} \in  C^1(\overline{D_k'})$ such that $ v_{k,n}^r  \to v_{k,n}$ in  $C^1(\overline{D_k'})$ as $r \to 0 $. Moreover,  $\underline{u}_{k,n} \leq v_{k,n} \leq u_0$ in $D_k'$ and $v_{k,n}$ is a solution of the problem
\begin{equation}
\label{514}
\left\{
\begin{array}{l}
-\Delta_pu=Rx_1^{-\eta}u^m-Qx_1^{-\gamma}u^q \,\,\, \mbox{in}\ D_k', \\
u>0~\mbox{in}~D_k',~~u=n\psi_k\ \ \mbox{on}\ \partial D_k'.
\end{array}
\right.
\end{equation}

Applying the Lemma \ref{L2}, we deduce that $v_{k,n}$ satisfies $ v_{k,n}\leq v_{k+1,n}$ and $v_{k,n}\leq u_0$ in $D_k'$. Thus, $v_{k,n} \to v_n$ in $C^1_{loc}(D)$, $v_n\leq u_0$ in $D$  and $v_n$ satisfies
 $$
\left\{
\begin{array}{l}
-\Delta_pu=Rx_1^{-\eta}u^m-Qx_1^{-\gamma}u^q \,\,\, \mbox{in}\ D, \\
u>0~\mbox{in}~D,~~u=n\ \ \mbox{on}\ \partial D.
\end{array}
\right.
$$
Using again  Lemma \ref{L2}, we derive that $v_{n}\leq v_{n+1}$ in $D$.

Finally, passing to the limit as $n \to +\infty$, we have that $v_n\to u_{min}$ in $C^1_{loc}(D)$. Besides this, following the arguments concerning to $u_{max}$, we show that
 $u_{min}$ is a minimal solution for (\ref{x1}), which depends on just  $x_1$. So, $u_{min}$ is a solution of problem (\ref{u1}) and from Lemma \ref{lemau1}, we have that
$$
u_{min}(x_1,x_2, \dots, x_n)=Ax_1^{-\alpha},~x_1>0 \,\,\, \mbox{and} \,\,\, (x_2, \dots, x_n) \in \mathbb{R}^{N-1},
$$
with $A>0$ being the unique  solution of (\ref{u2}).  Hence, given a $u\in C^1(D)$
solution of (\ref{x1}), we must to have
$$
u(x_1,x_2, \dots, x_n)=Ax_1^{-\alpha},~x_1>0 \,\,\, \mbox{and} \,\,\, (x_2, \dots, x_n) \in \mathbb{R}^{N-1}.
$$
This completes the proof of Lemma \ref{lemax1}. \fim
\end{proof}

\vspace{0.5 cm}

\section{Final comment} Here, we would like point out that Theorem \ref{T1} still holds, if in the assumption $(f_1)-(i)$ we have 
$$
\lim_{s\to0^+}{f(s)}/{s^{p-1}}<+\infty \,\,\, \mbox{if} \,\,\, a_0\geq0, 
$$
or  
$$
\lim_{s\to0^+}{f(s)}/{s^{p-1}}<{1}/{||b||_\infty} \,\,\, \mbox{if} \,\,\,  a_0<0. 
$$

\end{document}